\documentclass[12pt]{article}
\usepackage{xcolor}
\usepackage{graphicx}
\usepackage{amsmath}
\usepackage{graphicx}
\usepackage{amsfonts}
\usepackage{amssymb}
\usepackage{amsmath, amssymb ,amsthm, amsfonts, amsgen}
\usepackage{url}
\DeclareGraphicsExtensions{.eps,.bmp,.jpg,.pdf,.mps,.png,.gif}
\numberwithin{equation}{section} 
  \newtheorem{Theorem}{Theorem}[section]

\newtheorem{Proposition}[Theorem]{Proposition}
\newtheorem{Corollary}[Theorem]{Corollary}

\newtheorem{Remark}[Theorem]{Remark}

\newcommand\AMSname{AMS subject classifications}

\begin{document}

\title{A two-dimensional electrostatic model of interdigitated comb drive in longitudinal mode}
\author{Antonio Gaudiello\footnote{Dipartimento di Matematica e Applicazioni "Renato Caccioppoli",
Universit\`a degli Studi di Napoli Federico II, Via Cintia, Monte S.
Angelo, 80126 Napoli,  Italia. e-mail: gaudiell@unina.it}  and
Michel Lenczner\footnote{ UBFC / UTBM / FEMTO-ST, 26, Chemin de
l'Epitaphe, 25030 Besan\c{c}on, France. e-mail:
michel.lenczner@univ-fcomte.fr }}\date{ } \maketitle

\begin{abstract}

A periodic homogenization model of the electrostatic equation is constructed
for a  comb drive with a large number of fingers
and whose mode of operation is in-plane and longitudinal. The model is
obtained in the case where the distance between the rotor and the stator is
of an order $\varepsilon ^{\alpha }$, $\alpha \geq 2$, where $\varepsilon$ denotes the period of distribution of the fingers. The
model derivation uses the two-scale convergence technique.\ Strong convergences  are also established. This allows us to find, after a proper scaling,
the limit of the electrostatic force applied
to the rotor in the longitudinal direction.

\medskip

\par
\noindent2010 \AMSname:  35J05, 35B27
\end{abstract}

\section{Introduction}

The technology of Micro-Electro-Mechanical Systems, or MEMS, includes both
mechanical and electronic components on a single chip built with micro
fabrication techniques. The main MEMS parts are sensors, actuators, and
microelectronics. Many types of micro actuation techniques are available,
the most common of which are piezoelectric, magnetic, thermal,
electrochemical, and electrostatic actuation. The latter is clearly the most
widespread because of its compatibility with microfabrication technology,
its ease of integration and its low energy consumption. In particular,
electrostatic comb drives, introduced in \cite{tang1989laterally,
tang1990electrostatic} to enable large travel range at low driving voltage,
are among the most used electrostatically actuated devices in
microelectromechanical systems containing movable mechanical structures.

\begin{figure}[h]
    \centering
  \includegraphics[scale=0.25]{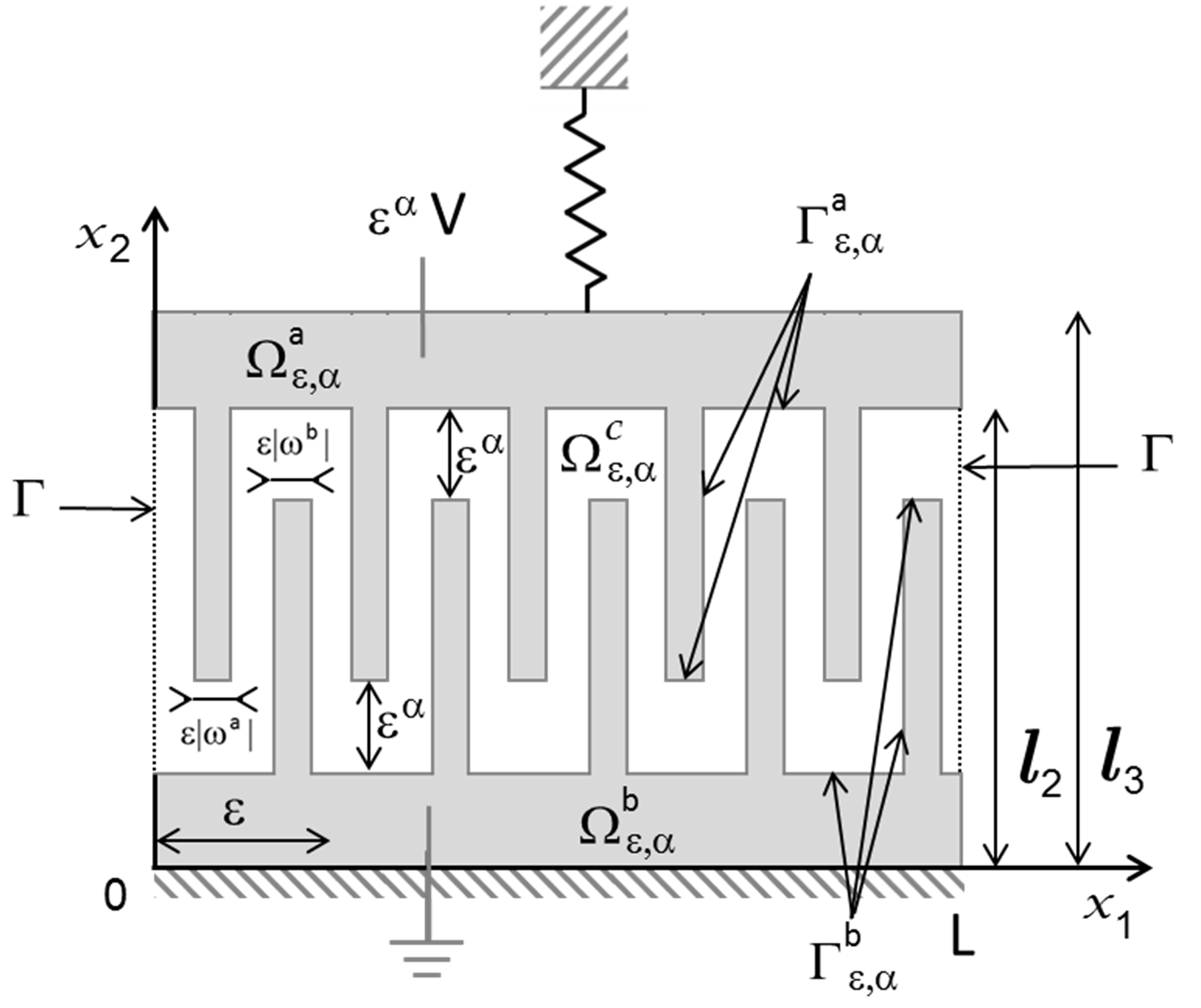}
  \caption{The   comb drive}
  \label{Fig1}
\end{figure}

A  comb drive is a deformable capacitor consisting of  conductive
stator and  rotor, each one composed of parallel fingers, that are
interdigitated, and whose number may exceed one hundred. The stator is
clamped and the rotor is suspended on elastic springs. The elastic
suspension is designed to allow the rotor to move in one of the desired
directions: longitudinal direction, i.e. parallel to the fingers, or in one of the two perpendicular
directions. From the electrical point of view, the stator is grounded and
the rotor is subjected to an electric potential $V$. The difference in
voltage induces an electrostatic force between the stator and the rotor
which causes a displacement of the rotor and therefore restoring forces in
the suspension. The equilibrium state is reached when the mechanical
restoring forces balance the electrostatic force.

The advantages of using electrostatic comb drive actuator approach include
low power dissipation, simple electronic control, and easy capacity-based
sensing mechanism. These devices are intended for applications in mechanical sensors, RF
communication, microbiology, mechanical power transmission,
long-range actuation, microphotonics, and microfluids
\cite{tang1990electrostatic, yeh1999integrated, kim1992silicon,
geiger1998new}. 

To achieve considerable electrostatic forces without
reverting to excessively high driving voltages, the freespace gap between the
electrodes must be minimal. With the advances of microfabrication
technology, thinner fingers and smaller gaps can be micromachined. This can
allow for a denser spacing of fingers and thus increase the power density of
comb drive actuators.

Design of complex MEMS involving multiple comb drives can not be performed
by trial and error due to the high microfabrication cost and time
consumption. Designers then make an intensive use of models. Part of the
comb drive modeling works focuse on the development of analytical models
that, beyond taking into account the electrostatic forces between parallel
parts, describe the fringe fields according to different methods and in many
configurations \cite{johnson1995electrophysics}, \cite{yeh2000electrostatic}%
, \cite{hammer2010analytical}, \cite{he2016analytical} \cite{li2012improved}%
, \cite{he2014calculating}, \cite{li2012improved}, and the analytical models
in the software package Coventor MEMS+ \cite{coventorMEMS+}.
On the other side, the use of direct numerical simulation remains the
reference approach for general configurations. Most often it is carried out
by a finite element method \cite{dong2011analysis}, \cite%
{chyuan2008computational}, \cite{ouakad2015numerical}, or a boundary
element method \cite{chyuan2004computational},
\cite{liao2004alternatively}. Despite the impressive increase of
computer power, the time scale required by their use for direct
simulation, optimization or calibration of complex systems is still
incompatible with the time scale of a designer.

Until now, the use of multiscale methods has not been yet explored on this
family of problems despite their periodic structure. However, they can offer
a good compromise between numerical methods adapted to general physics and
geometries, but expensive in simulation time, and analytical methods
developed for particular physics and geometries requiring only a few
computation resources.

 In this paper we develop a first comb drive multiscale model based on asymptotic methods. Precisely, we consider a 2-dimensional  model for an in-plane  comb drive, in a vacuum and in statical longitudinal regime, made by  a  rotor called $\Omega^a_{\varepsilon,\alpha}$ and a stator called $\Omega^b_{\varepsilon,\alpha}$ (see Figure \ref{Fig1}).
Both of them are composed by a set of $\varepsilon$-periodic fingers, with cross-section of order $\varepsilon$. The goal of this paper is to study the asymptotic behaviour of the  longitudinal electrostatic force applied on the rotor with respect  to two parameters: the period $\varepsilon$ and the small distance between the rotor and the stator.  {\it A priori} estimates show that in this model a  discriminating role is played by this distance that we consider of order $\varepsilon^\alpha$. Precisely, we prove that if $\alpha\geq2$ for obtaining asymptotically a force of order $O(1)$, the applied voltage has  to be of order $\varepsilon^\alpha V$ and in this case the  limit  force is given by
\begin{equation}\label{formulaprinciple}-\frac{\epsilon_0}{2}V^2L\left( \hbox{meas}(\omega^a)+\hbox{meas}(\omega^b)\right)\end{equation}
where  $\epsilon_0$ is the vacuum permittivity, $V$ is a constant  independent of $\varepsilon$, $L$  the comb length,  and $\hbox{meas}(\omega^a)$ and $\hbox{meas}(\omega^b)$  the  length of the cross section of the reference finger of the rotor and of the stator, respectively  (see Figure \ref{Fig1}). This result shows that only  the longitudinal forces on the extremities of the rotor's fingers  and on the part of the rotor's boundary  corresponding to the orthogonal projection of  the stator's fingers play a significant role. In particular, this means that the fringe field can be neglected in the  asymptotic regime $\alpha\geq2$. We expect that this phenomenon appears when $0\leq\alpha<2$. We also underline that in the limit force  there is no contribution of boundary layer effect on the lateral side of the comb, that are expected in other regimes.

The paper is organized in the following way.  The geometry of the comb drive is rigorously described in Section \ref{geometryapril2019}.  The problem satisfied by the electrical potential in  the vacuum between  the rotor and the stator is  given in Section \ref{Probbb} (see \eqref{J13,2019strong} where the voltage source is normalized by assuming  it equal to   1). The main result of this paper, i.e. the proof of formula  \eqref{formulaprinciple}, is stated in Theorem \ref{main theoremapril24,2019}. Section \ref{rescscsc} is devoted to rescale the problem given in Section \ref{Probbb} to a problem  on a domain where the finger's height  is  independent of $\varepsilon$ (see Figure \ref{Fig2}). Thus, the problem is split on three subdomains $\Omega^{c,1}_\varepsilon$,  $\Omega^{c,2}_\varepsilon$, and $\Omega^{c,3}_\varepsilon$ (see Figure \ref{Fig3}).
 Moreover, in  Proposition \ref{PropF25,2019}  we prove a key result which allows us to transform the longitudinal force applied on the rotor's boundary part  $\Gamma^a_{\varepsilon,\alpha}$  (see formula in \eqref{rafbarr} and also  p. 225 in \cite{kovetz2000electromagnetic}) into an  integral  on $\Omega^{c,1}_\varepsilon\cup\Omega^{c,2}_\varepsilon\cup\Omega^{c,3}_\varepsilon$.
{\it A priori} estimates of the rescaled solution of problem  \eqref{J13,2019strong} are obtained in Section \ref{apapstst}. They suggest  that different regimes depending on $\alpha$ can be  expected. Section \ref{casesssalpha=2} is devoted to prove Theorem \ref{main theoremapril24,2019}  in the case $\alpha=2$. The proof consists of several steps. In Section \ref{stime dettagliate alpha=2},  further {\it a priori} estimates of the rescaled solution are derived in the case $\alpha=2$. These estimates provide two-scale convergences (the two-scale convergence technique was proposed  in \cite{N} and developed  in
    \cite{A}, see also \cite{Casado},  \cite{CioDaGri}, and \cite{Lenc}). Then, in  Section \ref{wweeaakkconv} the two-scale  limits  are  identified on each subdomain $\Omega^{c,1}$,  $\Omega^{c,2}$, and $\Omega^{c,3}$ (see Figure \ref{Fig4}). The limit results are improved in Section \ref{corrrresult} by  corrector results. Finally in Section \ref{proofmmmmaitttheore}, these correctors allow us to pass to the limit in the  formula  of the longitudinal force stated in  Proposition \ref{PropF25,2019} and  to prove Theorem \ref{main theoremapril24,2019}  in the case $\alpha=2$.  The proof of Theorem \ref{main theoremapril24,2019}  in the case $\alpha>2$ is only sketched in Section \ref{sketched}.

Homogenization of oscillating boundaries with fixed amplitude is widely studied and we refer to the following main papers:  \cite{AiNaPr},    \cite{ki01ora},  \cite{ki01}, \cite{ki3}, \cite{BaCon}, \cite{ki9}, \cite{Bgg1}, \cite{BgM}, \cite{Blg}, \cite{BoOrRo}, \cite{BrChi},  \cite{BrChTh},  \cite{Chechtar}, \cite{DamPe}, \cite{Pan},  \cite{DeDuMe}, \cite{DeNan}, \cite{Durante1}, \cite{Durantebis}, \cite{EK}, \cite{gagu2}, \cite{gagumu}, \cite{gamel2}, \cite{GSili}, \cite{Lenczner} \cite{LencSmith}, \cite{m1}, \cite{m4},  \cite{NaPrSa}, \cite{Nk}, and \cite{Vinh}.

Also the homogenization of boundaries  with oscillations having small amplitude has a wide bibliography, but this argument is beyond the scope of this paper and a reader  interested in this subject can see some references quoted in  \cite{gagumu}.

\section{The geometry\label{geometryapril2019}}

Let $ \zeta_1,\,\zeta_2,\,\zeta_3,\,\zeta_4\in]0,1[$ be such that
$$\zeta_1<\zeta_2<\zeta_3<\zeta_4,$$
and set $$ \omega^a=]\zeta_1,\zeta_2[,\quad  \omega^b=]\zeta_3,\zeta_4[,$$
$$ \hbox{meas}(\omega^a)=\zeta_2-\zeta_1, \quad \hbox{meas}(\omega^b)=\zeta_4-\zeta_3.$$
\begin{figure}[h]
    \centering
  \includegraphics[scale=0.35]{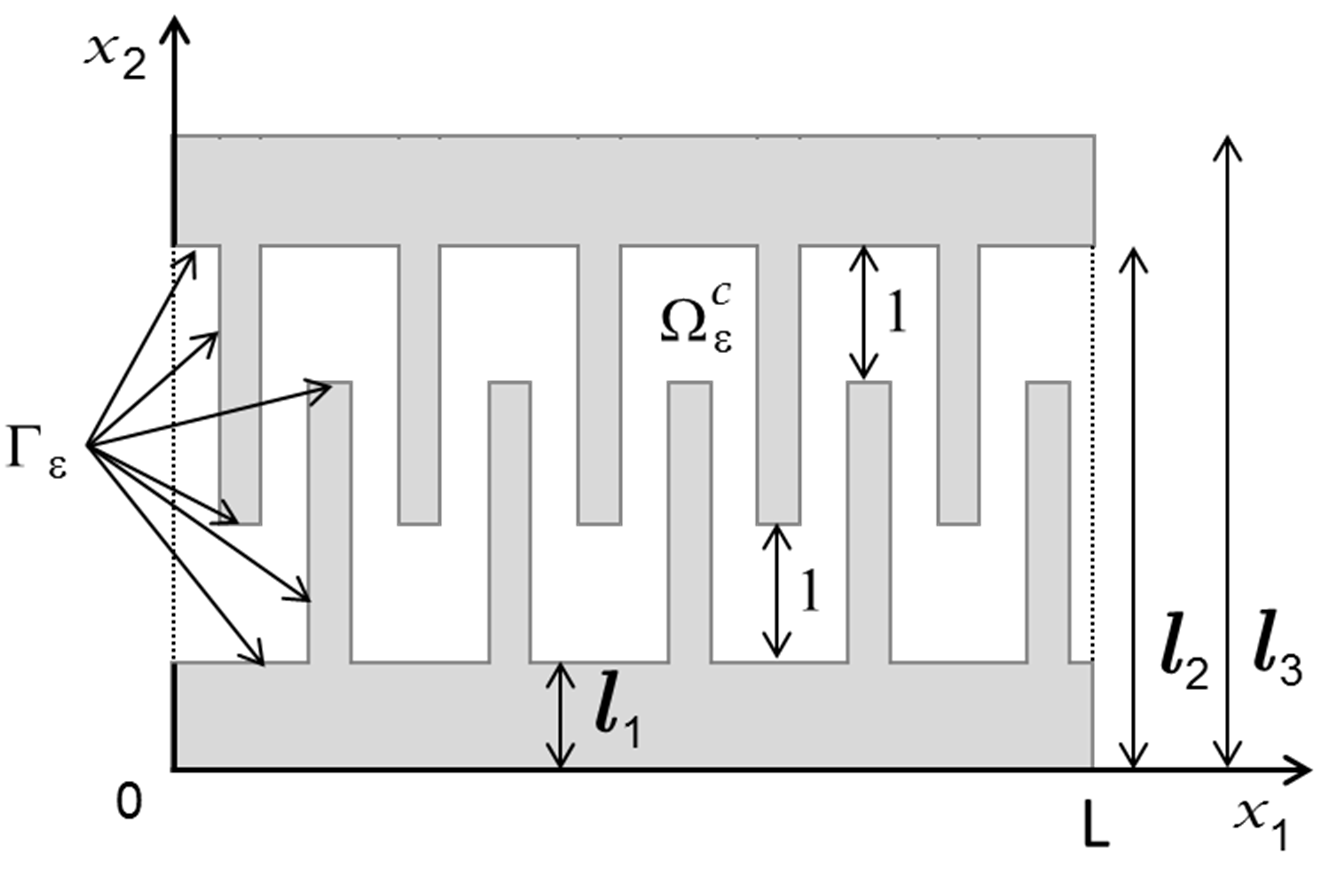}
  \caption{  The rescaled comb drive \label{Fig2}}
  \includegraphics[scale=0.35]{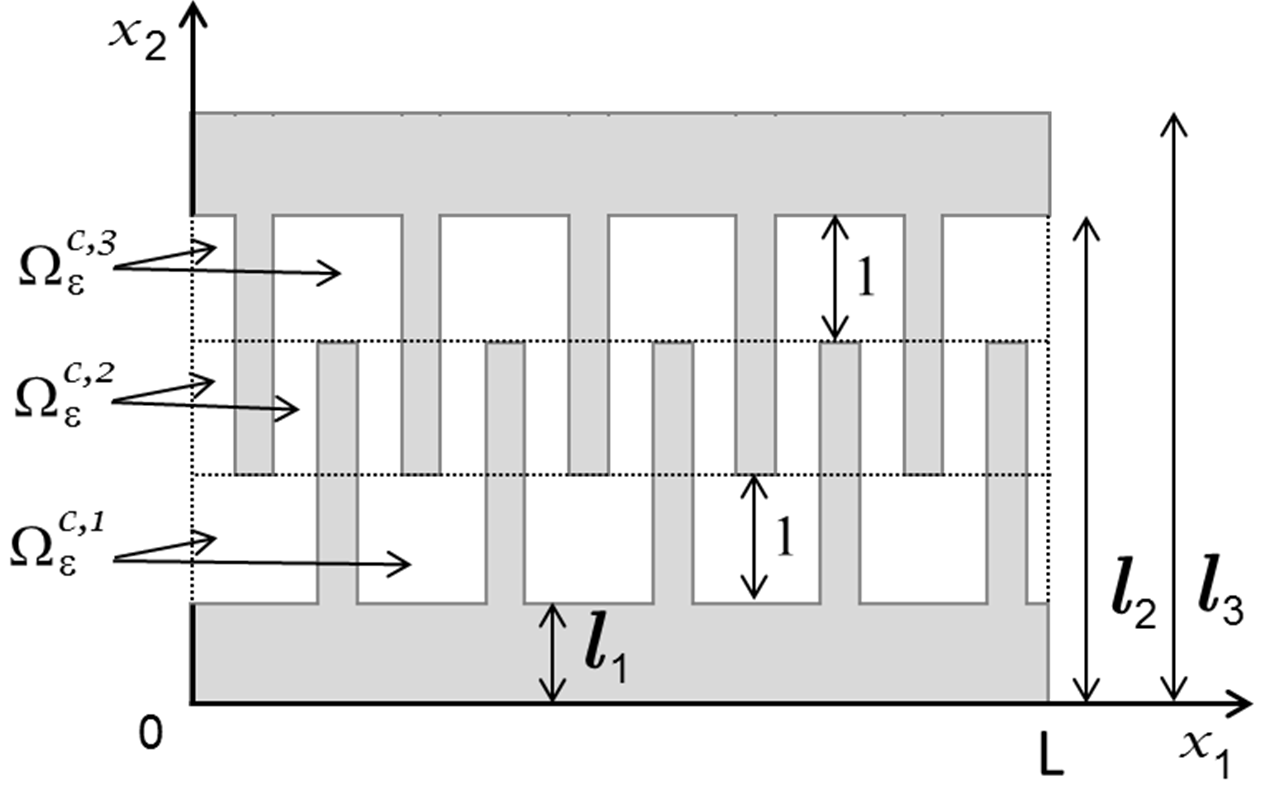}
  \caption{Decomposition of the  rescaled comb drive \label{Fig3}}
\end{figure}

Let $\alpha\in[0,+\infty[$, $L\in  ]0,+\infty[$, and $l_1,l_2,l_3\in  ]0,+\infty[$ be such that
 $$l_1+2<l_2<l_3.$$
For every  $\varepsilon\in \left\{\frac{L}{n}:n\in \mathbb{N}\right\}$ set (see Figure \ref{Fig1} for $\alpha>0$ or  Figure \ref{Fig2} for $\alpha=0$) 
 $$\Omega^a_{\varepsilon,\alpha}=\left(]0,L[\times]l_2,l_3[\right)\cup\left(\bigcup_{k=0}^{\frac{L}{\varepsilon}-1}\left(\varepsilon\omega^a+\varepsilon k\right)\times ]l_1+\varepsilon^\alpha,l_2]\right),$$ 
 $$\Omega^b_{\varepsilon,\alpha}=\left(]0,L[\times]0,l_1[\right)\cup\left(\bigcup_{k=0}^{\frac{L}{\varepsilon}-1}\left(\varepsilon\omega^b+\varepsilon k\right)\times [l_1,l_2-\varepsilon^\alpha[\right),$$
 $$ \Omega^c_{\varepsilon,\alpha}=\left(]0,L[\times]0,l_3[\right) \setminus\left( \overline{\Omega^a_{\varepsilon,\alpha}}\cup \overline{\Omega^b_{\varepsilon,\alpha}}\right),$$
 $$ \Gamma^a_{\varepsilon,\alpha}=\partial \Omega^a_{\varepsilon,\alpha}\cap\partial \Omega^c_{\varepsilon,\alpha},$$
 $$ \Gamma^b_{\varepsilon,\alpha}=\partial \Omega^b_{\varepsilon,\alpha}\cap\partial \Omega^c_{\varepsilon,\alpha},$$
 $$ \Gamma_{\varepsilon,\alpha}= \Gamma^a_{\varepsilon,\alpha}\cup  \Gamma^b_{\varepsilon,\alpha},$$
 $$\Gamma=\{0,L\}\times]l_1,l_2[.$$
 where $\Omega^a_{\varepsilon,\alpha}$ models the rotor, $\Omega^b_{\varepsilon,\alpha}$ the stator, each one composed of parallel fingers that are
interdigitated,  $\Omega^c_{\varepsilon,\alpha}$ the vacuum between the rotor and the stator, and $ \Gamma^a_{\varepsilon,\alpha}$ and  $ \Gamma^b_{\varepsilon,\alpha}$ are the parts of the boundary of the rotor and of the stator facing each other.
Moreover,  setting  (see Figure \ref{Fig3} for $\alpha=0$)
 $$ \Omega^{c,1}_{\varepsilon,\alpha}=   \Omega^c_{\varepsilon,\alpha}\cap\left(]0,L[\times ]l_1,l_1+\varepsilon^\alpha[\right),$$
 $$ \Omega^{c,2}_{\varepsilon,\alpha}=   \Omega^c_{\varepsilon,\alpha}\cap\left(]0,L[\times [l_1+\varepsilon^\alpha,l_2-\varepsilon^\alpha]\right),$$
  $$ \Omega^{c,3}_{\varepsilon,\alpha}=   \Omega^c_{\varepsilon,\alpha}\cap\left(]0,L[\times ]l_2-\varepsilon^\alpha, l_2[\right),$$
  the vacuum is split in three parts
  $$\Omega^c_{\varepsilon,\alpha}=\Omega^{c,1}_{\varepsilon,\alpha}\cup\Omega^{c,2}_{\varepsilon,\alpha}\cup\Omega^{c,3}_{\varepsilon,\alpha}.$$

Furthermore,  set (see Figure \ref{Fig4})
$$\Omega^{c,1} =]0,L[\times]l_1,l_1+1[,\quad \Omega^{c,2} =]0,L[\times]l_1+1,l_2-1[,\quad \Omega^{c,3} =]0,L[\times]l_2-1,l_2[.$$
\begin{Remark} For simplicity  we assumed $\varepsilon\in \left\{\frac{L}{n}:n\in \mathbb{N}\right\}$. Of course,  with small modifications in the proofs, all  results of this paper hold true with $\varepsilon\in ]0,1[$.
\end{Remark}

\begin{figure}[h]
    \centering
  \includegraphics[scale=0.35]{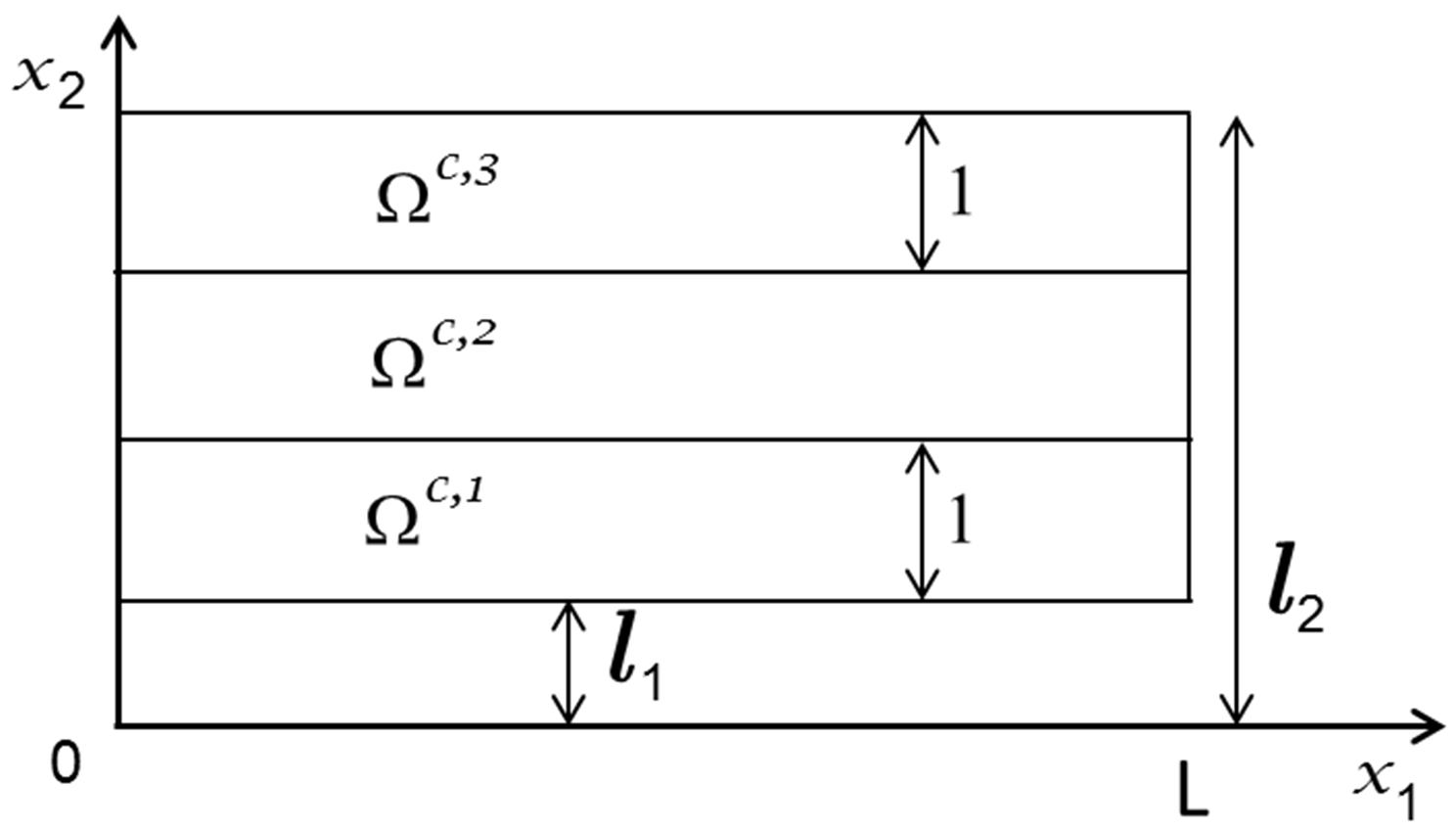}
  \caption{The limit domains}
  \label{Fig4}
\end{figure}

\section{The problem\label{Probbb}}
Let $\alpha\in [0,+\infty[$. Then,  for every $\varepsilon$ consider the following normalized problem
\begin{equation}\label{J13,2019strong}\left\{\begin{array}{lll}
-\Delta\phi_{\varepsilon}=0, \hbox{ in }  \Omega^c_{\varepsilon,\alpha},\\\\
\phi_{\varepsilon}= 1, \hbox{ on } \Gamma^a_{\varepsilon,\alpha},\\\\
\phi_{\varepsilon}= 0, \hbox{ on } \Gamma^b_{\varepsilon,\alpha},\\\\
\nabla \phi_{\varepsilon}\cdot\nu=0, \hbox{ on } \Gamma,
\end{array}\right.\end{equation}
where $\nu$ denotes  the unit normal 
to $\Gamma$ exterior to $ \Omega^c_{\varepsilon,\alpha}$.  The solution $\phi_{\varepsilon}$ represents the electrical potential  in the vacuum $ \Omega^c_{\varepsilon,\alpha}$ when the stator is grounded and the voltage in the rotor is assumed  equal to $1$. By setting
$$\mu_{\varepsilon,\alpha}=\left\{\begin{array}{ll}1,\hbox{ on }  \Gamma^a_{\varepsilon,\alpha},\\\\
0,\hbox{ on }  \Gamma^b_{\varepsilon,\alpha},\end{array}\right.$$
the weak formulation of \eqref{J13,2019strong} is 
\begin{equation}\label{J13,2019weak}\left\{\begin{array}{lll}\phi_{\varepsilon}\in H^1_{ \Gamma_{\varepsilon,\alpha}}(\Omega^c_{\varepsilon,\alpha},\mu_{\varepsilon,\alpha}),\\\\\displaystyle{\int_{\Omega^c_{\varepsilon,\alpha}}\nabla  \phi_{\varepsilon}\nabla \psi dx=0,\quad \forall\psi \in H^1_{ \Gamma_{\varepsilon,\alpha}}(\Omega^c_{\varepsilon,\alpha},0),}
\end{array}\right.\end{equation}
where for $g\in H^{-\frac{1}{2}} (\Gamma_{\varepsilon,\alpha})$ it is set $$H^1_{ \Gamma_{\varepsilon,\alpha}}(\Omega^c_{\varepsilon,\alpha},g)=\{\psi\in H^1(\Omega^c_{\varepsilon,\alpha}):\psi=g, \hbox{ on }\Gamma_{\varepsilon,\alpha}\}.$$

According to \cite{kovetz2000electromagnetic}, p. 225, the longitudinal  electrostatic force on rotor's boundary $\Gamma^a_{\varepsilon,\alpha}$  generated  by the electrical potential $\varepsilon^\alpha V \phi_{\varepsilon}$ in the vacuum is given by
\begin{equation}\label{rafbarr}-\frac{\epsilon_0}{2}V^2\int_{\Gamma^a_{\varepsilon,\alpha}} |\varepsilon^\alpha\nabla \phi_{\varepsilon}|^{2}\nu_2 \, ds,\end{equation}
where $\epsilon_0$ is the vacuum permittivity,$V$ is a constant independent of $\varepsilon$, and $\nu_2$ denotes the second component of the unit normal to $\Gamma^a_{\varepsilon,\alpha}$ exterior to $ \Omega^c_{\varepsilon,\alpha}$.

 The main result of this paper is the following one.
 \begin{Theorem}\label{main theoremapril24,2019} For every $\varepsilon$, let $\phi_\varepsilon$ be the unique solution to \eqref{J13,2019weak} with $\alpha\geq2$ and  let  $\nu_2$ denote the second component of the unit normal to $\Gamma^a_{\varepsilon,\alpha}$ exterior to $ \Omega^c_{\varepsilon,\alpha}$. Then,
\begin{equation}\label{F25,2019has}\lim_{\varepsilon\rightarrow 0}
\int_{\Gamma^a_{\varepsilon,\alpha}} |\varepsilon^\alpha\nabla \phi_{\varepsilon}|^{2}\nu_2 \, ds =L\left( \hbox{meas}(\omega^a)+\hbox{meas}(\omega^b)\right),\end{equation}
where $L$, $\omega^a$, and $\omega^b$ are defined in Section \ref{geometryapril2019}.
\end{Theorem}

In the sequel,  the dependence on $\alpha$ of the domain will be omitted  when $\alpha=0$. For instance, $\Omega^a_{\varepsilon,0}$ will be denoted by $\Omega^a_\varepsilon$, and so on.

\section{The rescaling\label{rescscsc}}
By virtue of  transformation (see Figure \ref{Fig1} and Figure \ref{Fig2})
\begin{equation}\label{J24,2019trasf} T_{\varepsilon,\alpha}:  \Omega^{c}_\varepsilon\rightarrow  \Omega^{c}_{\varepsilon,\alpha}
\end{equation}
defined by
\begin{equation}\label{J24,2019trasfpart}\left\{\begin{array}{ll} (x_1,x_2)\in \Omega^{c,1}_\varepsilon\rightarrow\left(x_1,(x_2-l_1)\varepsilon^\alpha+l_1\right)\in  \Omega^{c,1}_{\varepsilon,\alpha},\\\\
 (x_1,x_2)\in \Omega^{c,2}_\varepsilon\rightarrow \left(x_1,D_\varepsilon(x_2-l_1-1)+l_1+\varepsilon^\alpha\right)\in  \Omega^{c,2}_{\varepsilon,\alpha},\\\\
(x_1,x_2)\in \Omega^{c,3}_\varepsilon\rightarrow\left(x_1,(x_2-l_2+1)\varepsilon^\alpha+l_2-\varepsilon^\alpha\right)\in  \Omega^{c,3}_{\varepsilon,\alpha},\end{array} \right.\end{equation}with
\begin{equation}\label{J24,2019trasfpartcoff} D_\varepsilon=\frac{l_2-l_1-2\varepsilon^\alpha}{l_2-l_1-2},\end{equation}
problem \eqref{J13,2019weak} is rescaled in the following one
\begin{equation}\label{J13,2019weakrescaled}\left\{\begin{array}{lll} \varphi_\varepsilon\in H^1_{ \Gamma_\varepsilon}(\Omega^c_\varepsilon,\mu_\varepsilon),\\\\
\displaystyle{\int_{\Omega^{c,1}_\varepsilon\cup \Omega^{c,3}_\varepsilon}\left(  \varepsilon^\alpha \partial_{x_1}  \varphi_\varepsilon \partial_{x_1}\psi+\varepsilon^{-\alpha} \partial_{x_2}  \varphi_\varepsilon \partial_{x_2}\psi \right)dx}\\\\ \displaystyle{+
\int_{\Omega^{c,2}_\varepsilon}\left(  D_\varepsilon \partial_{x_1}  \varphi_\varepsilon \partial_{x_1}\psi+D^{-1}_\varepsilon\partial_{x_2}  \varphi_\varepsilon \partial_{x_2}\psi \right)dx=0, \quad\forall\psi \in H^1_{ \Gamma_\varepsilon}(\Omega^c_\varepsilon,0).
}
\end{array}\right.\end{equation}

Remark that 
\begin{equation}\label{J15,2019}\lim_{\varepsilon\rightarrow 0} D_\varepsilon= \frac{l_2-l_1}{l_2-l_1-2}.
\end{equation}

Let \begin{equation}\label{F16,2019z}\varphi^\star\in  C^\infty(\mathbb{R}\times[l_1,l_2])\end{equation} be such that
\begin{equation}\label{F16,2019z1}\left\{\begin{array}{ll}\varphi^\star(\cdot,x_2)\hbox{ is 1-periodic  for every } x_2\in [l_1,l_2],
\\\\ \varphi^\star=1,\hbox{  in }\omega^a\times]l_1+1,l_2[, \quad \varphi^\star=0,\hbox{  in }\omega^b\times]l_1,l_2-1[, \\\\
\varphi^\star=1,\hbox{ on }\mathbb{R}\times\{l_2\}, \quad \varphi^\star=0,\hbox{  on }\mathbb{R}\times\{l_1\},
\end{array}\right.\end{equation}
and for every $\varepsilon\in]0,1[$ set
\begin{equation}\label{F16,2019zq}\varphi^\star_\varepsilon(x_1, x_2)=\varphi^\star\left(\frac{x_1}{\varepsilon}, x_2\right), \hbox{  in }\mathbb{R}\times[l_1,l_2].\end{equation}

 The previous rescaling allows us to rewrite formula  \eqref{rafbarr}.
 \begin{Proposition}\label{PropF25,2019} For every $\varepsilon$, let $\phi_\varepsilon$ be the unique solution to \eqref{J13,2019weak}, $\varphi_\varepsilon$ be the unique solution to \eqref{J13,2019weakrescaled},  $ \varphi^\star_\varepsilon$ be defined by \eqref{F16,2019z}-\eqref{F16,2019zq}, $D_\varepsilon$ be defined in \eqref{J24,2019trasfpartcoff}, and let  $\nu_2$ denote the second component of the unit normal to $\Gamma^a_{\varepsilon,\alpha}$ exterior to $ \Omega^c_{\varepsilon,\alpha}$. Then, for every $\varepsilon$,
\begin{equation}\label{24Aprile2019serabis}\begin{array}{ll}\displaystyle{
\int_{\Gamma^a_{\varepsilon,\alpha}} |\nabla \phi_{\varepsilon} |^{2}\nu_2 ds=}\\\\
\displaystyle{
\int_{\Omega^{c,1}_\varepsilon\cup \Omega^{c,3}_\varepsilon} \left(-\partial _{x_{2}}
\varphi^\star_{\varepsilon }\left(\left \vert\partial_{x_1}  \varphi_\varepsilon\right\vert^2-\frac{1}{\varepsilon^{2\alpha}}\left\vert\partial_{x_2}  \varphi_\varepsilon\right\vert^2\right)+2\partial_{x_2}
\varphi_{\varepsilon }\partial_{x_1}
\varphi^\star_{\varepsilon }\partial_{x_1}
\varphi_{\varepsilon }
\right)dx}
\\\\
\displaystyle{+
\int_{\Omega^{c,2}_\varepsilon} \left(-\partial _{x_{2}}
\varphi^\star_{\varepsilon }\left(\left \vert\partial_{x_1}  \varphi_\varepsilon\right\vert^2-\frac{1}{D_\varepsilon^2}\left\vert\partial_{x_2}  \varphi_\varepsilon\right\vert^2\right)+2\partial_{x_2}
\varphi_{\varepsilon }\partial_{x_1}
\varphi^\star_{\varepsilon }\partial_{x_1}
\varphi_{\varepsilon }
\right)dx.}
\end{array}\end{equation}
\end{Proposition}
\begin{proof}
Let $T_{\varepsilon,\alpha}$ be defined by  \eqref{J24,2019trasf}-\eqref{J24,2019trasfpartcoff}. The first step is devoted to proving that
\begin{equation}\label{stylo0}\begin{array}{lll}\displaystyle{
\int_{\Gamma^a_{\varepsilon,\alpha}} |\nabla \phi_{\varepsilon} |^{2}\nu_2 ds}\\\\ \displaystyle{=\int_{\Omega _{\varepsilon ,\alpha }^{c}}\left(-\partial _{x_{2}}\left(
\varphi^\star_{\varepsilon }\circ T_{\varepsilon,\alpha}^{-1}\right)\left\vert\nabla \phi_{\varepsilon}\right\vert^2+2\partial _{x_{2}}\phi_{\varepsilon}  \nabla\left( \varphi^\star_{\varepsilon }\circ T_{\varepsilon,\alpha}^{-1}\right)\nabla
\phi_{\varepsilon}\right)dx, \quad\forall\varepsilon,}\end{array}
\end{equation}
from which \eqref{stylo0} follows by  changing of variable \eqref{J24,2019trasf} in the second integral.

As we shall show in the following, \begin{equation}\label{7890aprile2019}|\nabla \phi _{\varepsilon }|^{2}\in
W^{1,1}(\Omega _{\varepsilon ,\alpha }^{c}).\end{equation} In particular, also $(\varphi _{\varepsilon }^{\star }\circ T_{\varepsilon ,\alpha
}^{-1})|\nabla \phi _{\varepsilon }|^{2}$ belongs to $ W^{1,1}(\Omega _{\varepsilon
,\alpha }^{c})$. Thus, definitions \eqref{J24,2019trasf} and  \eqref{F16,2019z}-\eqref{F16,2019zq} allow us to write
\begin{equation}\begin{array}{ll} \label{stylo}\displaystyle{
\int_{\Gamma _{\varepsilon ,\alpha }^{a}}|\nabla \phi _{\varepsilon
}|^{2}\nu _{2}ds =\int_{\Gamma _{\varepsilon ,\alpha
}^{a}}(\varphi _{\varepsilon }^{\star }\circ T_{\varepsilon ,\alpha
}^{-1})|\nabla \phi _{\varepsilon }|^{2}\nu _{2}ds   }\\\\ \displaystyle{=\int_{\Gamma _{\varepsilon ,\alpha }\cup \Gamma }(\varphi _{\varepsilon
}^{\star }\circ T_{\varepsilon ,\alpha }^{-1})|\nabla \phi _{\varepsilon
}|^{2}\nu _{2}ds, \quad\forall\varepsilon.  }
\end{array}\end{equation}
The Green's Formula (for instance, see Th. 6.6-7 in \cite{Ciarlet}) gives
\begin{equation}  \label{stylo2}
\int_{\Gamma _{\varepsilon ,\alpha }\cup \Gamma }(\varphi _{\varepsilon
}^{\star }\circ T_{\varepsilon ,\alpha }^{-1})|\nabla \phi _{\varepsilon
}|^{2}\nu _{2}ds=\int_{\Omega _{\varepsilon ,\alpha }^{c}}\partial
_{x_{2}}((\varphi _{\varepsilon }^{\star }\circ T_{\varepsilon ,\alpha
}^{-1})|\nabla \phi _{\varepsilon }|^{2})dx , \quad\forall\varepsilon.
\end{equation}
 Then, \eqref{stylo} and \eqref{stylo2}) provides  
\begin{equation} \label{stylo3}\begin{array}{ll} \displaystyle{
\int_{\Gamma _{\varepsilon ,\alpha }^{a}}|\nabla \phi _{\varepsilon
}|^{2}\nu _{2}ds}\\\\ \displaystyle{=\int_{\Omega _{\varepsilon ,\alpha }^{c}}\partial
_{x_{2}}(\varphi _{\varepsilon }^{\star }\circ T_{\varepsilon ,\alpha
}^{-1})|\nabla \phi _{\varepsilon }|^{2}dx +2 \int_{\Omega _{\varepsilon ,\alpha }^{c}}(\varphi _{\varepsilon }^{\star
}\circ T_{\varepsilon ,\alpha }^{-1})\nabla \phi _{\varepsilon }\nabla (\partial
_{x_{2}} \phi _{\varepsilon })dx,  \quad\forall\varepsilon.}\end{array}
\end{equation}

On the other side (see below),  \begin{equation}\label{MiLenc}\nabla \phi _{\varepsilon } \in W^{1,\frac{3}{2}}(\Omega_{\varepsilon ,\alpha }^{c}).\end{equation} In particular,   $(\varphi _{\varepsilon }^{\star }\circ
T_{\varepsilon ,\alpha }^{-1})\nabla \phi _{\varepsilon } $ belongs to $ W^{1,\frac{3}{2}}(\Omega
_{\varepsilon ,\alpha }^{c})$, and  $\partial
_{x_{2}}\phi _{\varepsilon }$ belongs to  $ W^{1,\frac{3}{2}}(\Omega
_{\varepsilon ,\alpha }^{c})$ which is included in  $ W^{1,\frac{6}{5}}(\Omega
_{\varepsilon ,\alpha }^{c})$. Consequently, again applying the Green's Formula as it appears in
Theorem 6.6-7 in  \cite{Ciarlet} with exponents $p=\frac{3}{2}$ and $q=\frac{6}{5}$,  the last integral in the right-hand side of (\ref{stylo3}) becomes
\begin{equation}\label{stylo4}\begin{array}{ll} \displaystyle{
\int_{\Omega _{\varepsilon
,\alpha }^{c}}(\varphi _{\varepsilon }^{\star }\circ T_{\varepsilon ,\alpha
}^{-1})\nabla \phi _{\varepsilon }\nabla (\partial _{x_{2}}\phi_{\varepsilon })dx}\\\\
\displaystyle{
=-\int_{\Omega _{\varepsilon ,\alpha }^{c}}{div}((\varphi
_{\varepsilon }^{\star }\circ T_{\varepsilon ,\alpha }^{-1})\nabla \phi
_{\varepsilon })\text{ }\partial _{x_{2}}\phi _{\varepsilon }
dx+\int_{\Gamma _{\varepsilon ,\alpha }\cup \Gamma }(\varphi _{\varepsilon
}^{\star }\circ T_{\varepsilon ,\alpha }^{-1})\partial _{x_{2}}\phi
_{\varepsilon }\nabla \phi _{\varepsilon }\nu ds  ,  \quad\forall\varepsilon,
 }
\end{array}\end{equation}
where $\nu $ is the unit normal to $\Gamma _{\varepsilon ,\alpha }\cup
\Gamma $ exterior to $\Omega _{\varepsilon ,\alpha }^{c}$. Since%
\[
\int_{\Gamma _{\varepsilon ,\alpha }\cup \Gamma }(\varphi _{\varepsilon
}^{\star }\circ T_{\varepsilon ,\alpha }^{-1})\partial _{x_{2}}\phi
_{\varepsilon }\nabla \phi _{\varepsilon }\nu ds=\int_{\Gamma
_{\varepsilon ,\alpha }^{a}}|\nabla \phi _{\varepsilon }|^{2}\nu _{2}
ds, \quad\forall\varepsilon,
\]%
which can be  checked by inspectioning on each part of $\Gamma _{\varepsilon
,\alpha }\cup \Gamma $, one can rewrite (\ref{stylo4}) as
\begin{equation}\label{labeslao}\begin{array}{ll} \displaystyle{
\int_{\Omega _{\varepsilon ,\alpha }^{c}}(\varphi _{\varepsilon }^{\star
}\circ T_{\varepsilon ,\alpha }^{-1})\nabla \phi _{\varepsilon }\partial
_{x_{2}}\nabla \phi _{\varepsilon }dx}\\\\ \displaystyle{=-\int_{\Omega _{\varepsilon
,\alpha }^{c}}{div}((\varphi _{\varepsilon }^{\star }\circ
T_{\varepsilon ,\alpha }^{-1})\nabla \phi _{\varepsilon })\text{ }\partial
_{x_{2}}\phi _{\varepsilon }dx+\int_{\Gamma _{\varepsilon ,\alpha
}^{a}}|\nabla \phi _{\varepsilon }|^{2}\nu _{2}ds, \quad\forall\varepsilon.
 }
\end{array}\end{equation}

Comparing  (\ref{stylo3}) and \eqref{labeslao}  gives
\begin{equation}\nonumber\begin{array}{ll} \displaystyle{
\int_{\Gamma _{\varepsilon ,\alpha }^{a}}|\nabla \phi _{\varepsilon
}|^{2}\nu _{2}ds}\\\\ \displaystyle{ =-\int_{\Omega _{\varepsilon ,\alpha
}^{c}}\partial _{x_{2}}(\varphi _{\varepsilon }^{\star }\circ T_{\varepsilon
,\alpha }^{-1})|\nabla \phi _{\varepsilon }|^{2}dx+2\int_{\Omega
_{\varepsilon ,\alpha }^{c}}{div}((\varphi _{\varepsilon }^{\star
}\circ T_{\varepsilon ,\alpha }^{-1})\nabla \phi _{\varepsilon })\partial
_{x_{2}}\phi _{\varepsilon }dx} \\\\ \displaystyle{
=-\int_{\Omega _{\varepsilon ,\alpha }^{c}}\partial _{x_{2}}(\varphi
_{\varepsilon }^{\star }\circ T_{\varepsilon ,\alpha }^{-1})|\nabla \phi
_{\varepsilon }|^{2}dx+2\int_{\Omega _{\varepsilon ,\alpha
}^{c}}\nabla (\varphi _{\varepsilon }^{\star }\circ T_{\varepsilon ,\alpha
}^{-1})\nabla \phi _{\varepsilon }\partial _{x_{2}}\phi _{\varepsilon
}dx}\\\\
\displaystyle{+\int_{\Omega _{\varepsilon ,\alpha
}^{c}}(\varphi _{\varepsilon }^{\star }\circ T_{\varepsilon ,\alpha
}^{-1})\Delta \phi _{\varepsilon }\partial _{x_{2}}\phi _{\varepsilon }dx, \quad\forall\varepsilon,
 }
\end{array}\end{equation}
which provides (\ref{stylo0}) since $\Delta \phi _{\varepsilon }=0$ in $
\Omega _{\varepsilon ,\alpha }^{c}$.

Now, we sketch  the proof of \eqref{7890aprile2019}, based on the decomposition of $
\phi _{\varepsilon }$ as a sum of its singular and regular parts $\phi
_{\varepsilon }^{S}\in H^{1}(\Omega _{\varepsilon ,\alpha }^{c})$ and $\phi
_{\varepsilon }^{S}\in H^{2}(\Omega _{\varepsilon ,\alpha }^{c})$. At the
vicinity of any reentering corner with angle $\omega =\frac{3\pi }{2}$, the
expression in polar coordinate of the singular part reads $$\phi
_{\varepsilon }^{S}(r,\theta )=r^{\frac{2}{3}}\sin \left(\frac{2\theta }{3}\right).$$ 
Thus,
$$|\nabla \phi _{\varepsilon }^{S}|^{2}(r,\theta )=r^{-\frac{2}{3}}\Phi
_{0}(\theta ),$$  with $\Phi _{0}\in C^\infty$.
The expansion of $\nabla |\nabla \phi _{\varepsilon }|^{2}$ in $\phi
_{\varepsilon }^{S}$ and $\phi _{\varepsilon }^{R}$ includes four terms:\begin{equation}\label{nablabnabla}\nabla |\nabla \phi _{\varepsilon }^{S}|^{2}, \quad \nabla \nabla \phi _{\varepsilon }^{S}\nabla
\phi _{\varepsilon }^{R}, \quad \nabla |\nabla \phi _{\varepsilon }^{R}|^{2}, \hbox{ and }\nabla \nabla \phi _{\varepsilon }^{R}\nabla
\phi _{\varepsilon }^{S},\end{equation}
 of
which only the first  two  terms cause regularity problems.

 As the first term in \eqref{nablabnabla} is concerned, one has $$\nabla |\nabla \phi _{\varepsilon }^{S}|^{2}(r,\theta
)=r^{-\frac{5}{3}}\Phi _{1}(\theta ),$$ with $\Phi _{1}\in C^\infty$. Then, it is  integrable.  As the second term in \eqref{nablabnabla} is concerned, one has 
$$\nabla \nabla \phi _{\varepsilon }^{S}\nabla
\phi _{\varepsilon }^{R}=(r^{\frac{1}{3}}\nabla \nabla \phi _{\varepsilon
}^{S})(r^{-\frac{1}{3}}\nabla \phi _{\varepsilon }^{R})$$
and 
its integrability  comes from the
observation that both terms $r^{\frac{1}{3}}\nabla \nabla \phi _{\varepsilon }^{S}$
are $r^{-\frac{1}{3}}\nabla \phi _{\varepsilon }^{R}$ are  square
integrable. 

The contribution of the corners with mixed conditions, that is
at the ends of $\Gamma $, to the singular part is in $H^{2-\eta }(\Omega
_{\varepsilon ,\alpha }^{c})$ for any positive $\eta $ and does not yield
any regularity issue.

 Regularity result \eqref{MiLenc} can be  proved with the same  arguments.

\end{proof}

\section{{\it A priori} estimates\label{apapstst} }
\begin{Proposition} \label{J18,2019Proposition}For every $\varepsilon$, let $\varphi_\varepsilon$ be the unique solution to \eqref{J13,2019weakrescaled}. Then
\begin{equation}\label{J18,2019estimates}\exists c\in]0,+\infty[\quad:\quad\left\{\begin{array}{lll}\displaystyle{ \int_{\Omega^{c,1}_\varepsilon\cup \Omega^{c,3}_\varepsilon} \vert \partial_{x_1}  \varphi_\varepsilon\vert^2 dx\leq  c\left(\varepsilon^{-2-\alpha}+  \varepsilon^{-2\alpha}\right),}\\\\
\displaystyle{ \int_{\Omega^{c,1}_\varepsilon\cup \Omega^{c,3}_\varepsilon} \vert\partial_{x_2}  \varphi_\varepsilon \vert^2 dx\leq  c\left(\varepsilon^{\alpha-2}+  1\right), 
}\\\\ \displaystyle{ \int_{\Omega^{c,2}_\varepsilon}  \vert\nabla  \varphi_\varepsilon \vert^2 dx\leq  c\left(\varepsilon^{-2}+  \varepsilon^{-\alpha}\right),}
\end{array}\right.\quad\forall\varepsilon.\end{equation}
\end{Proposition}
\begin{proof} For every $\varepsilon$, let $\varphi^\star_\varepsilon$  be defined by \eqref{F16,2019z}-\eqref{F16,2019zq}.
Moreover, set 
$$Y=]0,1[ \times]l_1,l_2[.$$ Then, one has
\begin{equation} \label{J18,2019}\Vert \varphi^\star_\varepsilon\Vert^2_{L^2(\Omega^c_\varepsilon)}\leq\sum_{k=0}^{\frac{L}{\varepsilon}}\varepsilon\Vert \varphi^\star\Vert^2_{L^2(Y)}= L
\Vert \varphi^\star\Vert^2_{L^2(Y)}, \quad\forall\varepsilon.
\end{equation}
Similarly, one obtains
\begin{equation} \label{derx1}\Vert \partial_{x_1}\varphi^\star_\varepsilon\Vert^2_{L^2(\Omega^c_\varepsilon)}= \frac{L}{\varepsilon^2}
\Vert \partial_{x_1}\varphi^\star\Vert^2_{L^2(Y)}, \quad\forall\varepsilon,
\end{equation}
and
\begin{equation} \label{derx2}\Vert \partial_{x_2}\varphi^\star_\varepsilon\Vert^2_{L^2(\Omega^c_\varepsilon)}= L
\Vert \partial_{x_2}\varphi^\star\Vert^2_{L^2(Y)}, \quad\forall\varepsilon.
\end{equation}

Now choosing $\psi= \varphi_\varepsilon-\varphi^\star_\varepsilon$ as test function in \eqref{J13,2019weakrescaled} and using Young's inequality,  \eqref{J15,2019}, and estimates \eqref{derx1} and \eqref{derx2} provide
\begin{equation}\nonumber\begin{array}{lll}\displaystyle{\exists c\in]0,+\infty[\,\,:\int_{\Omega^{c,1}_\varepsilon\cup \Omega^{c,3}_\varepsilon}\left(  \varepsilon^\alpha\vert\partial_{x_1}  \varphi_\varepsilon\vert^2+\varepsilon^{-\alpha} \vert\partial_{x_2}  \varphi_\varepsilon \vert^2 \right)dx+
\int_{\Omega^{c,2}_\varepsilon}  \vert\nabla  \varphi_\varepsilon \vert^2dx
}\\\\ \leq  c\left(\varepsilon^{-2}+  \varepsilon^{-\alpha}\right), \quad\forall\varepsilon,
\end{array}\end{equation}
which implies \eqref{J18,2019estimates}.
\end{proof}

\section{The case $\alpha=2$\label{casesssalpha=2}}
This section is devoted to proving Theorem \ref{main theoremapril24,2019}  when  $\alpha=2$.
\subsection{{\it A priori} estimates \label{stime dettagliate alpha=2}}
Proposition \ref{J18,2019Proposition} immediately implies the following result.
\begin{Corollary} For every $\varepsilon$, let $\varphi_\varepsilon$ be the unique solution to \eqref{J13,2019weakrescaled} with $\alpha=2$. Then,
\begin{equation}\label{J18,2019estimatesalfa=2}\exists c\in]0,+\infty[\quad:\quad\left\{\begin{array}{lll}\displaystyle{ \Vert\varepsilon^2\partial_{x_1}  \varphi_\varepsilon\Vert_{L^2(\Omega^{c,1}_\varepsilon\cup \Omega^{c,3}_\varepsilon)}\leq  c,}\\\\
\displaystyle{  \Vert\partial_{x_2}  \varphi_\varepsilon \Vert_{L^2(\Omega^{c,1}_\varepsilon\cup \Omega^{c,3}_\varepsilon)}\leq  c, 
}\\\\ \displaystyle{ \Vert\varepsilon\nabla  \varphi_\varepsilon \Vert_{L^2(\Omega^{c,2}_\varepsilon)}\leq  c,}
\end{array}\right.\quad\forall\varepsilon.\end{equation}
\end{Corollary}
The next task is devoted to prove the following {\it a priori} estimate.
\begin{Proposition} \label{J21,2019prop} For every $\varepsilon$, let $\varphi_\varepsilon$ be the unique solution to \eqref{J13,2019weakrescaled} with $\alpha=2$. Then,
\begin{equation}\label{J18,2019last}\exists c\in]0,+\infty[\,\,:\Vert \varphi_\varepsilon  \Vert_{L^2(\Omega^c_\varepsilon)}\leq c, \quad\forall\varepsilon.
\end{equation}
\end{Proposition} 
\begin{proof}
The Dirichlet boundary condition of $\varphi_\varepsilon $ on $\Gamma_\varepsilon$ and the second estimate in \eqref{J18,2019estimatesalfa=2} provide that
\begin{equation}\nonumber\exists c\in]0,+\infty[\,\,:\Vert \varphi_\varepsilon  \Vert_{L^2(\Omega^{c,1}_\varepsilon\cup \Omega^{c,3}_\varepsilon)}\leq c, \quad\forall\varepsilon.
\end{equation}

The main task is   to prove that
\begin{equation}\label{J19,2019ultimo}\exists c\in]0,+\infty[\,\,:\Vert \varphi_\varepsilon  \Vert_{L^2(\Omega^{c,2}_\varepsilon)}\leq c, \quad\forall\varepsilon,\end{equation}
which completes the proof.
To this aim, set 
$$P=]0,1[\setminus\left(\overline{\omega^a}\cup\overline{\omega^b}\right)=]0,\zeta_1[\cup]\zeta_2,\zeta_3[\cup]\zeta_4,1[.$$
Fix $\varepsilon$. Then, one has
\begin{equation}\label{J19,2019}
\Vert \varphi_\varepsilon  \Vert^2_{L^2(\Omega^{c,2}_\varepsilon)}=\sum_{k=0}^{\frac{L}{\varepsilon}-1}\int_{\left(\varepsilon P+\varepsilon k\right)\times ]l_1,l_2[} \vert \varphi_\varepsilon  \vert^2 dx.\end{equation}
Now fix $k\in\left\{0,\cdots , \frac{L}{\varepsilon}-1\right\}$. Then, if $x_1\in \varepsilon P+\varepsilon k$, one   of the following three cases holds true:
$$x_1\in]\varepsilon k, \varepsilon\zeta_1+\varepsilon k[, \quad x_1\in]\varepsilon\zeta_2+\varepsilon k, \varepsilon\zeta_3+\varepsilon k[,\quad x_1\in]\varepsilon\zeta_4+\varepsilon  k,\varepsilon(1+k)[  .$$ In the first case, since
$$\varphi_\varepsilon=1,\hbox{ on }\{\varepsilon\zeta_1+\varepsilon k\}\times]l_1,l_2[,$$
one has
$$\varphi_\varepsilon(x_1,x_2)=1-\int_{x_1}^{\varepsilon\zeta_1+\varepsilon k}\partial_{x_1}\varphi_\varepsilon(t,x_2)dt,\quad\forall  x_1\in]\varepsilon k, \varepsilon\zeta_1+\varepsilon k[,\hbox{ for a.e. }x_2\in]l_1,l_2[,$$
which implies
\begin{equation}\label{J19,2019I} \int_{l_1}^{l_2}\int_{\varepsilon k}^{\varepsilon\zeta_1+\varepsilon k}\vert \varphi_\varepsilon(x_1,x_2)\vert^2dx_1dx_2\leq 2(l_2-l_1)\varepsilon+2\varepsilon^2  \int_{l_1}^{l_2}\int_{\varepsilon k}^{\varepsilon\zeta_1+\varepsilon k}\vert \partial_{x_1}\varphi_\varepsilon(x_1,x_2)\vert^2dx_1dx_2.
\end{equation}
Similarly,   since
$$\varphi_\varepsilon=0 , \hbox { on }\{\varepsilon\zeta_3+\varepsilon k\}\times]l_1,l_2[ \hbox { and on }\{\varepsilon\zeta_4+\varepsilon k\}\times]l_1,l_2[,$$
in the second and in the third  case
one has
\begin{equation}\label{J19,2019II} \int_{l_1}^{l_2}\int_{\varepsilon\zeta_2+\varepsilon k}^{\varepsilon\zeta_3+\varepsilon k}\vert \varphi_\varepsilon(x_1,x_2)\vert^2dx_1dx_2\leq 2\varepsilon^2 \int_{l_1}^{l_2}\int_{\varepsilon\zeta_2+\varepsilon k}^{\varepsilon\zeta_3+\varepsilon k} \vert \partial_{x_1}\varphi_\varepsilon(x_1,x_2)\vert^2dx_1dx_2
\end{equation}
and
\begin{equation}\label{J19,2019III} \int_{l_1}^{l_2}\int_{\varepsilon\zeta_4+\varepsilon k}^{\varepsilon(1+k)}\vert \varphi_\varepsilon(x_1,x_2)\vert^2dx_1dx_2\leq 2\varepsilon^2\int_{l_1}^{l_2}\int_{\varepsilon\zeta_4+\varepsilon k}^{\varepsilon(1+k)} \vert \partial_{x_1}\varphi_\varepsilon(x_1,x_2)\vert^2dx_1dx_2.
\end{equation}
Adding \eqref{J19,2019I}, \eqref{J19,2019II}, and \eqref{J19,2019III} gives
\begin{equation}\nonumber
\int_{\left(\varepsilon P+\varepsilon k\right)\times ]l_1,l_2[} \vert \varphi_\varepsilon  \vert^2 dx\leq 2(l_2-l_1)\varepsilon+2\varepsilon^2 \int_{\left(\varepsilon P+\varepsilon k\right)\times ]l_1,l_2[} \vert \partial_{x_1}\varphi_\varepsilon  \vert^2 dx,\end{equation}
from which, summing up $k\in\left\{0,\cdots , \frac{L}{\varepsilon}-1\right\}$ and using \eqref{J19,2019} and  the third estimate in \eqref{J18,2019estimatesalfa=2}, one obtains \eqref{J19,2019ultimo}.
\end{proof}
\subsection{Weak convergence results\label{wweeaakkconv}}
The next proposition is devoted to  studying the limit  in $ \Omega^{c,2}$, as $\varepsilon$ tends to zero,  of problem \eqref{J13,2019weakrescaled} with $\alpha=2$.
\begin{Proposition} \label{Proposizione Monda3}For every $\varepsilon$, let $\varphi_\varepsilon$ be the unique solution to \eqref{J13,2019weakrescaled} with $\alpha=2$. Set  $$\varphi_{\varepsilon,2}={\varphi_\varepsilon}_{|_{\Omega_\varepsilon^{c,2}}}$$  and
\begin{equation}\label{F13,2019}\overline{\varphi_{\varepsilon,2}}=\left\{\begin{array}{ll}\varphi_{\varepsilon,2},\hbox{ a.e. in }\Omega_\varepsilon^{c,2},\\\\ \displaystyle{1,\hbox{ a.e.  in }\bigcup_{k=0}^{\frac{L}{\varepsilon}-1}\left(\varepsilon\omega^a+\varepsilon k\right)\times ]l_1+1,l_2-1[,}\\\\ \displaystyle{0, \hbox{ a.e. in }\bigcup_{k=0}^{\frac{L}{\varepsilon}-1}\left(\varepsilon\omega^b+\varepsilon k\right)\times ]l_1+1,l_2-1[.}\end{array}\right.\end{equation}
Let 
\begin{equation}\label{J24,2019all} \varphi_2:y\in[0,1]\longrightarrow \left\{\begin{array}{ll}\dfrac{y+1-\zeta_4}{\zeta_1-\zeta_4+1}, \hbox{ if }y\in[0,\zeta_1],\\\\
1, \hbox{ if }y\in[\zeta_1,\zeta_2],\\\\ \dfrac{y-\zeta_3}{\zeta_2-\zeta_3}, \hbox{ if }y\in[\zeta_2,\zeta_3],\\\\
0, \hbox{ if }y\in[\zeta_3,\zeta_4],\\\\\dfrac{y-\zeta_4}{\zeta_1-\zeta_4+1}, \hbox{ if }y\in[\zeta_4,1].
\end{array}\right.
\end{equation}
Then, 
\begin{equation}\label{Monda1bisterforse}\left\{\begin{array}{ll}\overline{\varphi_{\varepsilon,2}}\hbox{ two scale converges to } \varphi_2,\\\\
\varepsilon\partial_{x_1}\overline{\varphi_{\varepsilon,2}}\hbox{ two scale converges to }\partial_{y} \varphi_2,\\\\
\varepsilon\partial_{x_2}\overline{\varphi_{\varepsilon,2}}\hbox{ two scale converges to }0,\end{array}\right.\end{equation}
as $\varepsilon$ tends to zero.
\end{Proposition}
\begin{proof}
Proposition \ref{J21,2019prop} and the third estimate in \eqref{J18,2019estimatesalfa=2} ensure the existence of a subsequence of $\{\varepsilon\}$, still denoted by $\{\varepsilon\}$,  and $u_2 \in L^2\left(\Omega^{c,2}, H^1_{\hbox{per}}(]0,1[)\right)$   (in possible dependence on the subsequence) such that
 \begin{equation}\label{Monda1bister}\left\{\begin{array}{ll}\overline{\varphi_{\varepsilon,2}}\hbox{ two scale converges to } u_2,\\\\
\varepsilon\partial_{x_1}\overline{\varphi_{\varepsilon,2}}\hbox{ two scale converges to }\partial_{y} u_2,\\\\
\varepsilon\partial_{x_2}\overline{\varphi_{\varepsilon,2}}\hbox{ two scale converges to }0,\end{array}\right.\end{equation}
as $\varepsilon$ tends to zero.

The next step is devoted to proving  that
 \begin{equation}\label{J23,2019fgd} u_2=1, \hbox { a.e. in } \Omega^{c,2}\times\omega^a.\end{equation}
Indeed, the definition of  $\overline{ \varphi_{\varepsilon,2}}$ gives
\begin{equation}\label{chiacchiereIIIchie}\begin{array}{ll}\displaystyle{ \int_{\Omega^{c,2}}\overline{\varphi_{\varepsilon,2}}(x_1,x_2)\psi \left(x_1,x_2,\frac{x_1}{\varepsilon}\right) dx_1dx_2=  \int_{\Omega^{c,2}}\psi \left(x_1,x_2,\frac{x_1}{\varepsilon}\right) dx_1dx_2,}  \\\\
\forall\psi\in C_0^\infty(\Omega^{c,2}\times\omega^a),\quad\forall\varepsilon.
\end{array}
\end{equation}
Passing to the limit, as $\varepsilon$ tends to zero, in \eqref{chiacchiereIIIchie} and using  the first limit in \eqref{Monda1bister} provide
\begin{equation}\nonumber\begin{array}{ll}\displaystyle{ \int_{\Omega^{c,2}\times\omega^a}u_2(x_1,x_2,y)\psi \left(x_1,x_2,y\right) dx_1dx_2dy=  \int_{\Omega^{c,2}\times\omega^a}\psi \left(x_1,x_2,y\right) dx_1dx_2dy, }\\\\\forall\psi\in C_0^\infty(\Omega^{c,2}\times\omega^a),
\end{array}
\end{equation}
which implies \eqref{J23,2019fgd}.

\noindent Similarly, one proves that
\begin{equation}\label{J23,2019dfg} u_2=0, \hbox { a.e. in } \Omega^{c,2}\times\omega^b.\end{equation} Finally, choosing $\displaystyle{\psi=\varepsilon^2\chi_1(x_1,x_2)\chi_2\left(\frac{x_1}{\varepsilon}\right)}$ with $\chi_1\in C_0^\infty\left(\Omega^{c,2}\right)$ and $\chi_2\in H^1_{\hbox{per}}\left(]0,1[\right)$ such that $\chi_2=0$ in $\omega^a\cup\omega^b$ as test function in \eqref{J13,2019weakrescaled} with $\alpha=2$ gives 
\begin{equation}\label{J23,2019seraIIJ24}\begin{array}{lll} \displaystyle{D_\varepsilon \varepsilon^2\int_{\Omega^{c,2}}  \partial_{x_1} \overline{ \varphi_{\varepsilon,2}}\left(\partial_{x_1}\chi_1(x_1,x_2)\chi_2\left(\frac{x_1}{\varepsilon}\right)+\varepsilon^{-1}\chi_1(x_1,x_2)\partial_{y}\chi_2\left(\frac{x_1}{\varepsilon}\right)\right)dx_1dx_2} \\\\
\displaystyle{+D_\varepsilon^{-1}\varepsilon^2\int_{\Omega^{c,2}} \partial_{x_2}  \overline{\varphi_{\varepsilon,2} }\partial_{x_2}\chi_1(x_1,x_2)\chi_2\left(\frac{x_1}{\varepsilon}\right)dx_1dx_2=0,}\\\\ \forall \chi_1\in C_0^\infty\left(\Omega^{c,2}\right), \quad\forall \chi_2\in H^1_{\hbox{per}}\left(]0,1[\right)\,\,:\,\, \chi_2=0,\hbox{ in } \omega^a\cup\omega^b,\quad\forall\varepsilon.
\end{array}\end{equation}
Passing to the limit, as $\varepsilon$ tends to zero, in \eqref{J23,2019seraIIJ24} and using  the second and third limits in   \eqref{Monda1bister}, and \eqref{J15,2019}
provide that, for a.e. $(x_1,x_2)$ in $ \Omega^{c,2}$,
\begin{equation}\label{J23,2019seraIJ24}\begin{array}{ll}\displaystyle{ \int_{]0,1[\setminus\left(\omega^a\cup\omega^b\right)}\partial_{y}u_2(x_1,x_2,y)\partial_{y}\chi_2(y)dy=0, }\\\\  \forall \chi_2\in H^1_{\hbox{per}}\left(]0,1[\right)\,\,:\,\, \chi_2=0,\hbox{ in } \omega^a\cup\omega^b.
\end{array}
\end{equation}
Problem \eqref{J23,2019fgd}, \eqref{J23,2019dfg}, and \eqref{J23,2019seraIJ24} is equivalent to the following problem independent of $(x_1,x_2)$
\begin{equation}\left\{\begin{array}{ll}\partial^2_{y^2}u_2=0,\hbox{ in }]0,1[\setminus\left(\omega^a\cup\omega^b\right),\\\\
u_2=1, \hbox {  in } \omega^a,\\\\
u_2=0, \hbox {  in } \omega^b,\\\\
  u_2(0)=u_2(1),\\\\
   \partial_y u_2(0)=\partial_yu_2(1),
\end{array}\right.\end{equation}
which admits  \eqref{J24,2019all}  as unique solution.
Consequently, limits in  \eqref{Monda1bister} hold for the whole sequence and \eqref{Monda1bisterforse} is satisfied.
\end{proof}

The next proposition is devoted to  studying the limit in $\Omega^{c,3}$ and in $\Omega^{c,1}$, as $\varepsilon$ tends to zero,  of problem \eqref{J13,2019weakrescaled} with $\alpha=2$.

\begin{Proposition} \label{Proposizione Eliquis1}For every $\varepsilon$, let $\varphi_\varepsilon$ be the unique solution to \eqref{J13,2019weakrescaled} with $\alpha=2$. Set  $$\varphi_{\varepsilon,3}={\varphi_\varepsilon}_{|_{\Omega_\varepsilon^{c,3}}},\quad \varphi_{\varepsilon,1}={\varphi_\varepsilon}_{|_{\Omega_\varepsilon^{c,1}}},$$  
\begin{equation}\label{Eliquis2}\widetilde{\varphi_{\varepsilon,3}}\left\{\begin{array}{ll}\varphi_{\varepsilon,3},\hbox{ a.e.  in }\Omega_\varepsilon^{c,3},\\\\ 1,\hbox{  a.e. in }\Omega^{c,3}\setminus \Omega_\varepsilon^{c,3},\end{array}\right.\end{equation}
and
\begin{equation}\label{Eliquis4}\widehat{\varphi_{\varepsilon,1}}=\left\{\begin{array}{ll}\varphi_{\varepsilon,1},\hbox{ a.e. in }\Omega_\varepsilon^{c,1},\\\\ 0,\hbox{ a.e. in }\Omega^{c,1}\setminus \Omega_\varepsilon^{c,1}.\end{array}\right.\end{equation}
Moreover, let
\begin{equation}\label{F12,2019}\varphi_3:(x_1,x_2,y)\in \Omega^{c,3}\times]0,1[\longrightarrow\left\{\begin{array}{ll}x_2+1-l_2, \hbox{ if } y\in \omega^b,\\\\1, \hbox{ if } y\in]0,1[\setminus \omega^b,\end{array}\right.
\end{equation}
and
\begin{equation}\label{F12,2019bis}\varphi_1:(x_1,x_2,y)\in \Omega^{c,1}\times]0,1[\longrightarrow\left\{\begin{array}{ll}x_2-l_1, \hbox{ if } y\in \omega^a,\\\\0, \hbox{ if } y\in]0,1[\setminus \omega^a.\end{array}\right.
\end{equation}
Then
\begin{equation}\label{Monda1bisterforseter}\left\{\begin{array}{ll}\widetilde{\varphi_{\varepsilon,3}}\hbox{ two scale converges to } \varphi_3,\\\\
\partial_{x_2}\widetilde{\varphi_{\varepsilon,3}}\hbox{ two scale converges to }\partial_{x_2} \varphi_3,\end{array}\right.\end{equation}
and
\begin{equation}\label{Monda1bisterforsequater}\left\{\begin{array}{ll}\widehat{\varphi_{\varepsilon,1}}\hbox{ two scale converges to } \varphi_1,\\\\
\partial_{x_2}\widehat{\varphi_{\varepsilon,1}}\hbox{ two scale converges to }\partial_{x_2} \varphi_1,\end{array}\right.\end{equation}
as $\varepsilon$ tends to zero.
\end{Proposition}
\begin{proof}  The proof will be  developed in several steps.

Proposition \ref{J21,2019prop} and the second estimate in \eqref{J18,2019estimatesalfa=2} ensure the existence of a subsequence of $\{\varepsilon\}$, still denoted by $\{\varepsilon\}$,  $u_3$,  $\xi\in L^2(\Omega^{c,3}\times]0,1[)$, and $w$, $z\in L^2(]0,L[\times]0,1[)$  (in possible dependence on the subsequence) satisfying 
\begin{equation} \label{F15,2019a} \widetilde{\varphi_{\varepsilon,3}}\hbox{ two scale converges to } u_3,
\end{equation}
and
\begin{equation}\label{J21,2019I}\left\{ \begin{array}{ll}
\partial_{x_2}\widetilde{\varphi_{\varepsilon,3}}\hbox{ two scale converges to }\xi,\\\\ \hbox{ the trace of }
\widetilde{{\varphi_{\varepsilon,3}}}   \hbox{ on } ]0,L[\times\{l_2-1\}\hbox{ two scale converges to } w,
\\\\ \hbox{ the trace of }
\widetilde{{\varphi_{\varepsilon,3}}}   \hbox{ on } ]0,L[\times\{l_2\}\hbox{ two scale converges to } z,
\end{array}\right.
\end{equation}
as $\varepsilon$ tends to zero.

The first step is devoted to proving that
\begin{equation}\label{J21,2019II}\xi=\partial_{x_2} u_3, \hbox { a.e. in } \Omega^{c,3}\times]0,1[.
\end{equation}
Indeed, integration by parts gives
\begin{equation}\label{chiacchiereI}\begin{array}{ll}\displaystyle{ \int_{\Omega^{c,3}}\partial_{x_2}\widetilde{\varphi_{\varepsilon,3}}(x_1,x_2)\psi \left(x_1,x_2,\frac{x_1}{\varepsilon}\right) dx_1dx_2}\\\\
\displaystyle{ =- \int_{\Omega^{c,3}}\widetilde{\varphi_{\varepsilon,3}}(x_1,x_2)\partial_{x_2}\psi \left(x_1,x_2,\frac{x_1}{\varepsilon}\right) dx_1dx_2,\quad\forall\psi\in C_0^\infty(\Omega^{c,3}\times]0,1[),\quad\forall\varepsilon.}
\end{array}
\end{equation}
Passing to the limit, as $\varepsilon$ tends to zero, in \eqref{chiacchiereI} and using  \eqref{F15,2019a} and the first  limit in \eqref{J21,2019I}  provide
\begin{equation}\nonumber\begin{array}{ll}\displaystyle{ \int_{\Omega^{c,3}\times]0,1[}\xi(x_1,x_2,y)\psi (x_1,x_2,y) dx_1dx_2 dy}\\\\\displaystyle{=-
\int_{\Omega^{c,3}\times]0,1[} u_3 (x_1,x_2,y) \partial_{x_2}\psi (x_1,x_2,y) dx_1dx_2 dy, \quad\forall\psi\in C_0^\infty(\Omega^{c,3}\times]0,1[),
}
\end{array}
\end{equation}
which implies \eqref{J21,2019II}. Combining the first limit in \eqref{J21,2019I}  with \eqref{J21,2019II}
gives
 \begin{equation}\label{Monda2}\partial_{x_2}\widetilde{\varphi_{\varepsilon,3}}\hbox{ two scale converges to }\partial_{x_2} u_3,\end{equation}
as $\varepsilon$ tends to zero.

The fact that $u_3$ and  $\xi\in L^2(\Omega^{c,3}\times]0,1[)$ combined with  \eqref{J21,2019II} provides that for a.e. $y\in ]0,1[$ $u_3(\cdot,\cdot,y)$ has  traces on $]0,l[\times\{l_2-1\}$ and on $]0,l[\times\{l_2\}$ belonging to $L^2(]0,l[\times\{l_2-1\})$ and to $L^2(]0,l[\times\{l_2\})$, respectively. 
The second step is devoted to proving that
\begin{equation}\label{J22,2019}w(x_1,y)=u_3(x_1, l_2-1,y),  \hbox{  a.e  in }  ]0,L[\times ]0,1[.
\end{equation}
Indeed,  integration by parts gives 
\begin{equation}\label{chiacchiereII}\begin{array}{ll}\displaystyle{ \int_{\Omega^{c,3}}\partial_{x_2}\widetilde{\varphi_{\varepsilon,3}}(x_1,x_2)\psi \left(x_1,\frac{x_1}{\varepsilon}\right) (l_2-x_2)dx_1dx_2}\\\\
\displaystyle{ = \int_{\Omega^{c,3}}\widetilde{\varphi_{\varepsilon,3}}(x_1,x_2)\psi \left(x_1,\frac{x_1}{\varepsilon}\right) dx_1dx_2}
\displaystyle{ - \int_{]0.L[}\widetilde{\varphi_{\varepsilon,3}}(x_1,l_2-1)\psi \left(x_1,\frac{x_1}{\varepsilon}\right) dx_1,}\\\\
\forall\psi\in C_0^\infty(]0,L[\times]0,1[),\quad\forall\varepsilon.
\end{array}
\end{equation}
Passing to the limit, as $\varepsilon$ tends to zero, in \eqref{chiacchiereII} and using
  \eqref{F15,2019a},  the second  limit in \eqref{J21,2019I}, and  \eqref{Monda2}  provide
\begin{equation}\nonumber\begin{array}{ll}\displaystyle{ \int_{\Omega^{c,3}\times]0,1[}\partial_{x_2}u_3(x_1,x_2,y)\psi \left(x_1,y\right)(l_2-x_2) dx_1dx_2dy}\\\\
\displaystyle{ = \int_{\Omega^{c,3}\times]0,1[}u_3(x_1,x_2,y)\psi \left(x_1,y\right) dx_1dx_2dy}
\displaystyle{ - \int_{]0.L[\times]0,1[}w(x_1,y)\psi \left(x_1,y\right) dx_1dy,}\\\\
\forall\psi\in C_0^\infty(]0,L[\times]0,1[),
\end{array}
\end{equation}
that is
\begin{equation}\nonumber\begin{array}{ll}\displaystyle{ \int_{]0,L[\times]0,1[}w(x_1,y)\psi \left(x_1,y\right) dx_1dy,= \int_0^1\left(\int_0^L w(x_1,y)\psi(x_1,y)dx_1\right)dy=}\\\\ \displaystyle{ \int_0^1\left(\int_{\Omega^{c,3}} \left( u_3(x_1,x_2,y)\psi(x_1,y) -\partial_{x_2}u_3(x_1,x_2,y)\psi(x_1,y)(l_2-x_2) \right) dx_1dx_2\right)dy}\\\\
\displaystyle{=\int_0^1\left(\int_0^L u_3(x_1,l_2-1,y) \psi(x_1,y)dx_1\right)dy=\int_{]0,L[\times]0,1[} u_3(x_1,l_2-1,y) \psi \left(x_1,y\right) dx_1dy,}\\\\\forall\psi\in C_0^\infty(]0,L[\times]0,1[),
\end{array}
\end{equation}
which implies \eqref{J22,2019}.
Similarly, one proves that
\begin{equation}\label{J22,2019bis}z(x_1,y)=u_3(x_1, l_2,y),  \hbox{  a.e  in }  ]0,L[\times ]0,1[.
\end{equation}

The third step is devoted to proving that
\begin{equation}\label{J23,2019I} u_3(x_1,l_2-1, y)=0,\hbox { a.e. in } ]0,L[\times\omega^b,\end{equation}
Indeed, the boundary condition of $\varphi_\varepsilon$ on $\Gamma^b_\varepsilon$ gives
\begin{equation}\label{chiacchiereIV}\begin{array}{ll}\displaystyle{ \int_{]0,L[}\widetilde{\varphi_{\varepsilon,3}}(x_1,l_2-1)\psi \left(x_1,\frac{x_1}{\varepsilon}\right) dx_1=  0, \quad\forall\psi\in C_0^\infty(]0,L[\times\omega^b),\quad\forall\varepsilon.}
\end{array}
\end{equation}
Passing to the limit, as $\varepsilon$ tends to zero, in \eqref{chiacchiereIV} and using the second limit in \eqref{J21,2019I}  and \eqref{J22,2019} provide
\begin{equation}\nonumber\begin{array}{ll}\displaystyle{ \int_{]0,L[\times\omega^b}u_3(x_1,l_2-1,y)\psi \left(x_1,y\right) dx_1dy=  0, \quad\forall\psi\in C_0^\infty(]0,L[\times\omega^b),}
\end{array}
\end{equation}
which implies \eqref{J23,2019I}. Similarly, one proves 
\begin{equation}\label{J23,2019III} u_3(x_1,l_2, y)=1,\hbox { a.e. in } ]0,L[\times]0,1[.\end{equation}

Arguing as in the proof of \eqref{J23,2019fgd} gives
\begin{equation}\label{J23,2019} u_3=1, \hbox { a.e. in } \Omega^{c,3}\times\omega^a,
\end{equation}

The fourth step is devoted to proving that 
 \begin{equation}\label{J23,2019seraI}\begin{array}{ll}\displaystyle{ \int_{\Omega^{c,3}\times\left(]0,1[\setminus\omega^a\right)}\partial_{x_2}u_3(x_1,x_2,y)\partial_{x_2}\chi(x_1,x_2,y)dx_1dx_2dy=0,}\\\\ \forall\chi\in C_0^\infty\left(\Omega^{c,3}\times\left(]0,1[\setminus\omega^a\right)\right).\end{array}
\end{equation}
Indeed, choosing $\displaystyle{\psi=\varepsilon^2\chi\left(x_1,x_2,\frac{x_1}{\varepsilon}\right)}$ with $\chi\in C_0^\infty\left(\Omega^{c,3}\times\left(]0,1[\setminus\omega^a\right)\right)$ as test function in \eqref{J13,2019weakrescaled} with $\alpha=2$ gives 
\begin{equation}\label{J23,2019seraII}\begin{array}{lll} \displaystyle{\int_{\Omega^{c,3}} \varepsilon^4 \partial_{x_1} \widetilde{ \varphi_{\varepsilon,3}}\left(\partial_{x_1}\chi\left(x_1,x_2,\frac{x_1}{\varepsilon}\right)+\varepsilon^{-1}\partial_{y}\chi\left(x_1,x_2,\frac{x_1}{\varepsilon}\right)\right)dx_1dx_2} \\\\
\displaystyle{+\int_{\Omega^{c,3}}  \partial_{x_2}  \widetilde{\varphi_{\varepsilon,3} }\partial_{x_2}\chi\left(x_1,x_2,\frac{x_1}{\varepsilon}\right)dx_1dx_2=0,}\quad\forall\chi\in C_0^\infty\left(\Omega^{c,3}\times\left(]0,1[\setminus\omega^a\right)\right)),\quad\forall\varepsilon.
\end{array}\end{equation}
Passing to the limit, as $\varepsilon$ tends to zero, in \eqref{J23,2019seraII} and using the first  estimate in \eqref{J18,2019estimatesalfa=2}, \eqref{Monda2}, and \eqref{J23,2019} provide \eqref{J23,2019seraI}.

In a similar way, one proves that
there exist  a subsequence of $\{\varepsilon\}$, still denoted by $\{\varepsilon\}$ and  $u_1 \in L^2(\Omega^{c,1}\times]0,1[)$  (in possible dependence on the subsequence)   such that
\begin{equation}\label{Monda1bis}\widehat{\varphi_{\varepsilon,1}}\hbox{ two scale converges to } u_1,\end{equation}
as $\varepsilon$ tends to zero. Moreover,  $\partial_{x_2}u_1\in  L^2(\Omega^{c,1}\times]0,1[)$ and
\begin{equation}\label{Monda2bis}\partial_{x_2}\widehat{\varphi_{\varepsilon,1}}\hbox{ two scale converges to }\partial_{x_2} u_1,\end{equation}
as $\varepsilon$ tends to zero. Furthermore,
\begin{equation}\label{J23,2019bis} u_1=0, \hbox { a.e. in } \Omega^{c,1}\times\omega^b,
\end{equation}
\begin{equation}\label{J23,2019IIbis} u_1(x_1,l_1+1, y)=1,\hbox { a.e. in } ]0,L[\times\omega^a,\end{equation}
\begin{equation}\label{J23,2019IIIbis} u_1(x_1,l_1, y)=0,\hbox { a.e. in } ]0,L[\times]0,1[,\end{equation}
and
 \begin{equation}\label{J23,2019seraIbis}\begin{array}{ll}\displaystyle{ \int_{\Omega^{c,1}\times\left(]0,1[\setminus\omega^b\right)}\partial_{x_2}u_1(x_1,x_2,y)\partial_{x_2}\chi(x_1,x_2,y)dx_1dx_2dy=0,}\\\\ \forall\chi\in C_0^\infty\left(\Omega^{c,1}\times\left(]0,1[\setminus\omega^b\right)\right).\end{array}
\end{equation}

The last step is devoted to proving that
\begin{equation}\label{F5,2019}\begin{array}{ll}\displaystyle{ \int_{\Omega^{c,3}\times\left(]0,1[\setminus\left(\omega^a\cup\omega^b\right)\right)}\partial_{x_2}u_3(x_1,x_2,y)\partial_{x_2}\chi(x_1,x_2,y)dx_1dx_2dy}\\\\ \displaystyle{ +\int_{\Omega^{c,1}\times\left(]0,1[\setminus\left(\omega^a\cup\omega^b\right)\right)}\partial_{x_2}u_1(x_1,x_2,y)\partial_{x_2}\chi(x_1,x_2,y)dx_1dx_2dy,} \\\\ \forall\chi\in C_0^\infty\left(\Omega^c\times\left(]0,1[\setminus\left(\omega^a\cup\omega^b\right)\right)\right).\end{array}
\end{equation}
Indeed, choosing $\displaystyle{\psi=\varepsilon^2\chi\left(x_1,x_2,\frac{x_1}{\varepsilon}\right)}$ with $\chi\in C_0^\infty\left(\Omega^c\times\left(]0,1[\setminus\left(\omega^a\cup\omega^b\right)\right)\right)$ as test function in \eqref{J13,2019weakrescaled} with $\alpha=2$ gives 

\begin{equation}\label{F2,2019}\begin{array}{lll} \displaystyle{\int_{\Omega^{c,3}} \varepsilon^4 \partial_{x_1} \widetilde{ \varphi_{\varepsilon,3}}\left(\partial_{x_1}\chi\left(x_1,x_2,\frac{x_1}{\varepsilon}\right)+\varepsilon^{-1}\partial_{y}\chi\left(x_1,x_2,\frac{x_1}{\varepsilon}\right)\right)dx_1dx_2} \\\\
\displaystyle{+\int_{\Omega^{c,3}}  \partial_{x_2}  \widetilde{\varphi_{\varepsilon,3} }\partial_{x_2}\chi\left(x_1,x_2,\frac{x_1}{\varepsilon}\right)dx_1dx_2}\\\\
\displaystyle{+\int_{ \Omega^{c,1}} \varepsilon^4 \partial_{x_1} \widehat{ \varphi_{\varepsilon,3}}\left(\partial_{x_1}\chi\left(x_1,x_2,\frac{x_1}{\varepsilon}\right)+\varepsilon^{-1}\partial_{y}\chi\left(x_1,x_2,\frac{x_1}{\varepsilon}\right)\right)dx_1dx_2} \\\\
\displaystyle{+\int_{ \Omega^{c,1}}  \partial_{x_2}  \widehat{\varphi_{\varepsilon,3} }\partial_{x_2}\chi\left(x_1,x_2,\frac{x_1}{\varepsilon}\right)dx_1dx_2}\\\\
\displaystyle{+D_\varepsilon \varepsilon^2\int_{\Omega^{c,2}}  \partial_{x_1} \overline{ \varphi_{\varepsilon,2}}\left(\partial_{x_1}\chi\left(x_1,x_2,\frac{x_1}{\varepsilon}\right)+\varepsilon^{-1}\partial_{y}\chi\left(x_1,x_2,\frac{x_1}{\varepsilon}\right)\right)dx_1dx_2} \\\\
\displaystyle{+D_\varepsilon^{-1}\varepsilon^2\int_{\Omega^{c,2}} \partial_{x_2}  \overline{\varphi_{\varepsilon,2} }\partial_{x_2}\chi\left(x_1,x_2,\frac{x_1}{\varepsilon}\right)dx_1dx_2=0,}\\\\ \forall\chi\in C_0^\infty\left(\Omega^c\times\left(]0,1[\setminus\left(\omega^a\cup\omega^b\right)\right)\right),\quad\forall\varepsilon.
\end{array}\end{equation}
Passing to the limit, as $\varepsilon$ tends to zero, in \eqref{F2,2019} and using the first estimate in \eqref{J18,2019estimatesalfa=2}, \eqref{Monda2}, \eqref{J23,2019}, \eqref{Monda2bis}, \eqref{J23,2019bis}, \eqref{J15,2019},  and the second and third limit in \eqref{Monda1bisterforse} provide
\begin{equation}\label{F5,2019Pasqua2019}\begin{array}{ll}\displaystyle{ \int_{\Omega^{c,3}\times\left(]0,1[\setminus\left(\omega^a\cup\omega^b\right)\right)}\partial_{x_2}u_3(x_1,x_2,y)\partial_{x_2}\chi(x_1,x_2,y)dx_1dx_2dy}\\\\ \displaystyle{ +\int_{\Omega^{c,1}\times\left(]0,1[\setminus\left(\omega^a\cup\omega^b\right)\right)}\partial_{x_2}u_1(x_1,x_2,y)\partial_{x_2}\chi(x_1,x_2,y)dx_1dx_2dy+}\\\\ \displaystyle{ \frac{l_2-l_1}{l_2-l_1-2}. \int_{\Omega^{c,2}\times\left(]0,1[\setminus\left(\omega^a\cup\omega^b\right)\right)}
\partial_{y}\varphi_2(x_1,x_2,y)\partial_{y}\chi(x_1,x_2,y)dx_1dx_2dy,
}
 \\\\ \forall\chi\in C_0^\infty\left(\Omega^c\times\left(]0,1[\setminus\left(\omega^a\cup\omega^b\right)\right)\right).\end{array}
\end{equation}
which implies  \eqref{F5,2019}, since the last integral in \eqref{F5,2019Pasqua2019}  is zero due to   \eqref{J24,2019all}.

Finally, \eqref{J23,2019I},  \eqref{J23,2019III}, \eqref{J23,2019}, \eqref{J23,2019seraI}, and \eqref{J23,2019bis}-\eqref{F5,2019} assert that $u_3$ and $u_1$  solve the following problems
\begin{equation}\label{F5bis',2019}\left\{\begin{array}{ll}u_3=1, \hbox { in } \Omega^{c,3}\times\omega^a,\\\\  \left\{\begin{array}{ll}\partial^2_{x^2_2}u_3 (x_1,x_2,y)=0, \hbox{ in }\Omega^{c,3}\times\left(]0,1[\setminus\omega^a\right),\\\\
u_3(x_1,l_2, y)=1,\hbox { in } ]0,L[\times]0,1[,\\\\
u_3(x_1,l_2-1, y)=0,\hbox { in } ]0,L[\times\omega^b,\\\\
\partial_{x_2} u_3(x_1,l_2-1, y)=0,\hbox { in } ]0,L[\times]0,1[\setminus\left(\omega^a\cup\omega^b\right),\end{array}\right.
\end{array}\right.
\end{equation}
and
\begin{equation}\label{F5ter",2019}\left\{\begin{array}{ll}u_1=0, \hbox { in } \Omega^{c,1}\times\omega^b,\\\\  \left\{\begin{array}{ll}
\partial^2_{x^2_2}u_1 (x_1,x_2,y)=0, \hbox{ in }\Omega^{c,1}\times\left(]0,1[\setminus\omega^b\right),\\\\
u_1(x_1,l_1, y)=0,\hbox { in } ]0,L[\times]0,1[,\\\\
u_1(x_1,l_1+1, y)=1,\hbox {  in } ]0,L[\times\omega^a,\\\\
\partial_{x_2} u_1(x_1,l_1+1, y)=0,\hbox { in } ]0,L[\times]0,1[\setminus\left(\omega^a\cup\omega^b\right),\end{array}\right.
\end{array}\right.
\end{equation}
respectively, which means that $u_3$ and $u_1$ are given by  \eqref{F12,2019} and \eqref{F12,2019bis}, respectively.  Consequently, \eqref{F15,2019a}, \eqref{Monda2}, \eqref{Monda1bis}, and \eqref{Monda2bis} hold true for the whole sequence and  \eqref{Monda1bisterforseter} and \eqref{Monda1bisterforsequater} are satisfied.
\end{proof}

The following result is an immediate consequence of Proposition \ref{Proposizione Monda3} and Proposition \ref{Proposizione Eliquis1}.
\begin{Corollary}
For every $\varepsilon$, let $\varphi_\varepsilon$ be the unique solution to \eqref{J13,2019weakrescaled} with $\alpha=2$  and let  $\overline{\varphi_{\varepsilon,2}}$,
$\widetilde{\varphi_{\varepsilon,3}}$, and $\widehat{\varphi_{\varepsilon,1}}$ be defined by \eqref{F13,2019}, \eqref{Eliquis2}, and  \eqref{Eliquis4}, respectively. Moreover,  let  $\varphi_2$,  $\varphi_3$,  and $\varphi_1$ be defined by \eqref{J24,2019all}, \eqref{F12,2019}, and \eqref{F12,2019bis}, respectively. Then
\begin{equation} \nonumber\overline{\varphi_{\varepsilon,2}}\rightharpoonup \frac{1}{2}\left(1+\hbox{meas}(\omega^a)-\hbox{meas}(\omega^b)\right),\quad\varepsilon\partial_{x_1}\overline{\varphi_{\varepsilon,2}}\rightharpoonup0,\quad \varepsilon\partial_{x_2}\overline{\varphi_{\varepsilon,2}}\rightharpoonup0, \hbox{ weakly in }L^2(\Omega^{c,2}),
\end{equation}
\begin{equation} \nonumber\widetilde{\varphi_{\varepsilon,3}}\rightharpoonup (x_2-l_2)\hbox{meas}(\omega^b)+1, \quad\partial_{x_2}\widetilde{\varphi_{\varepsilon,3}}\rightharpoonup\hbox{meas}(\omega^b),\hbox{ weakly in }L^2(\Omega^{c,3}),
\end{equation}
and
\begin{equation} \nonumber\widehat{\varphi_{\varepsilon,1}}\rightharpoonup (x_2-l_1)\hbox{meas}(\omega^a),\quad  \quad\partial_{x_2}\widehat{\varphi_{\varepsilon,1}}\rightharpoonup\hbox{meas}(\omega^a),\hbox{ weakly in }L^2(\Omega^{c,1}),
\end{equation}
as $\varepsilon$ tends to zero.
\end{Corollary}

\subsection{Corrector results\label{corrrresult}}
Th following proposition is devoted to proving the energies convergence.
\begin{Proposition}\label{encon2019} For every $\varepsilon$, let $\varphi_\varepsilon$ be the unique solution to \eqref{J13,2019weakrescaled} with $\alpha=2$. Moreover,  let   $\varphi_1$,  $\varphi_3$, and $\varphi_2$,  be defined by  \eqref{F12,2019bis},  \eqref{F12,2019}, and \eqref{J24,2019all},  respectively. Then
\begin{equation}\label{F15,2019energyconv}\begin{array}{lll} 
\displaystyle{\lim_{\varepsilon\rightarrow 0}\bigg[\int_{\Omega^{c,1}_\varepsilon\cup \Omega^{c,3}_\varepsilon}\left( \left\vert  \varepsilon^2\partial_{x_1}  \varphi_\varepsilon \right\vert^2+\left\vert\partial_{x_2}  \varphi_\varepsilon \right\vert^2\right)dx} \\\\ \displaystyle{+
\int_{\Omega^{c,2}_\varepsilon}\left(  D_\varepsilon \left\vert\varepsilon\partial_{x_1}  \varphi_\varepsilon \right\vert^2+D^{-1}_\varepsilon\left\vert\varepsilon\partial_{x_2}  \varphi_\varepsilon \right\vert^2 \right)dx\bigg]
}\\\\
\displaystyle{=
\int_{\Omega^{c,1}\times\omega^a}\left\vert\partial_{x_2}  \varphi_1\right\vert^2dxdy+\int_{\Omega^{c,3}\times\omega^b}\left\vert\partial_{x_2}  \varphi_3\right\vert^2dxdy}\\\\ \displaystyle{+ \frac{l_2-l_1}{l_2-l_1-2} \int_{\Omega^{c,2}\times\left(]0,1[\setminus\left(\omega^a\cup\omega^b\right)\right)}\left\vert\partial_{y}  \varphi_2\right\vert^2dxdy.}
\end{array}\end{equation}
\end{Proposition}
\begin{proof}
Choosing $\psi= \varepsilon^2\left(\varphi_\varepsilon-\varphi^\star_\varepsilon\right)$ as test function in \eqref{J13,2019weakrescaled}, where $\varphi^\star_\varepsilon$ is defined by \eqref{F16,2019z}-\eqref{F16,2019zq}, gives
\begin{equation}\label{F16,2019tk}\begin{array}{lll} 
\displaystyle{\int_{\Omega^{c,1}_\varepsilon\cup \Omega^{c,3}_\varepsilon}\left( \left\vert  \varepsilon^2\partial_{x_1}  \varphi_\varepsilon \right\vert^2+\left\vert\partial_{x_2}  \varphi_\varepsilon \right\vert^2\right)dx+
\int_{\Omega^{c,2}_\varepsilon}\left(  D_\varepsilon  \left\vert  \varepsilon\partial_{x_1}  \varphi_\varepsilon \right\vert^2+D^{-1}_\varepsilon \left\vert  \varepsilon\partial_{x_2}  \varphi_\varepsilon \right\vert^2\right)dx
}\\\\ 
\displaystyle{=\int_{\Omega^{c,1}}\left(  \varepsilon^3 \partial_{x_1}  \widehat{\varphi_{\varepsilon,1}}\left(\partial_{y}\varphi^\star\right)\left(\frac{x_1}{\varepsilon}, x_2\right)+\partial_{x_2}  \widehat{\varphi_{\varepsilon,1}}\partial_{x_2}\varphi^\star\left(\frac{x_1}{\varepsilon}, x_2\right) \right)dx}\\\\
\displaystyle{+\int_{\Omega^{c,3}}\left(  \varepsilon^3 \partial_{x_1} \widetilde{\varphi_{\varepsilon,3}}\left(\partial_{y}\varphi^\star\right)\left(\frac{x_1}{\varepsilon}, x_2\right)+\partial_{x_2} \widetilde{\varphi_{\varepsilon,3}}\partial_{x_2}\varphi^\star\left(\frac{x_1}{\varepsilon}, x_2\right)\right)dx}\\\\ \displaystyle{+
\int_{\Omega^{c,2}}\left(  D_\varepsilon \varepsilon \partial_{x_1} \overline{\varphi_{\varepsilon,2}}\left(\partial_{y}\varphi^\star\right)\left(\frac{x_1}{\varepsilon}, x_2\right)+D^{-1}_\varepsilon \varepsilon^2\partial_{x_2}  \overline{\varphi_{\varepsilon,2}} \partial_{x_2}\varphi^\star\left(\frac{x_1}{\varepsilon}, x_2\right)\right)dx,\quad\forall\varepsilon,}
\end{array}\end{equation}
where    
$\widehat{\varphi_{\varepsilon,1}}$, $\widetilde{\varphi_{\varepsilon,3}}$, $\overline{\varphi_{\varepsilon,2}}$  are defined by    \eqref{Eliquis4},  \eqref{Eliquis2},  and \eqref{F13,2019}, respectively.
Passing to the limit, as $\varepsilon$ tends to zero, in \eqref{F16,2019tk} and using  \eqref{J15,2019}, the first estimate in \eqref{J18,2019estimatesalfa=2}, Proposition \ref{Proposizione Monda3}, and  Proposition \ref{Proposizione Eliquis1} provide
\begin{equation}\label{F17,2019tkk}\begin{array}{lll} 
\displaystyle{\lim_{\varepsilon\rightarrow 0}\bigg[\int_{\Omega^{c,1}_\varepsilon\cup \Omega^{c,3}_\varepsilon}\left( \left\vert  \varepsilon^2\partial_{x_1}  \varphi_\varepsilon \right\vert^2+\left\vert\partial_{x_2}  \varphi_\varepsilon \right\vert^2\right)dx}\\\\ \displaystyle{+
\int_{\Omega^{c,2}_\varepsilon}\left(  D_\varepsilon  \left\vert  \varepsilon\partial_{x_1}  \varphi_\varepsilon \right\vert^2+D^{-1}_\varepsilon \left\vert  \varepsilon\partial_{x_2}  \varphi_\varepsilon \right\vert^2\right)dx\bigg]
}\\\\
\displaystyle{=\int_{\Omega^{c,1}\times\omega^a}\partial_{x_2}\varphi^\star dxdy+\int_{\Omega^{c,3}\times\omega^b} \partial_{x_2}\varphi^\star dxdy}
\\\\ \displaystyle{+ \frac{l_2-l_1}{l_2-l_1-2} 
\int_{\Omega^{c,2}\times\left(]0,1[\setminus\left(\omega^a\cup\omega^b\right)\right)}   \partial_{y} \varphi_2\partial_{y}\varphi^\star dxdy.}
\end{array}\end{equation}
As the third integral and fourth integral in \eqref{F17,2019tkk} are concerned, the last two lines in \eqref{F16,2019z1},    \eqref{F12,2019}, and  \eqref{F12,2019bis} ensure that
\begin{equation}\label{F17,2019tkkha}\begin{array}{lll} 
\displaystyle{\int_{\Omega^{c,1}\times\omega^a}\partial_{x_2}\varphi^\star dxdy+\int_{\Omega^{c,3}\times\omega^b} \partial_{x_2}\varphi^\star dxdy=\int_{\Omega^{c,1}\times\omega^a} 1dxdy+\int_{\Omega^{c,3}\times\omega^b} 1dxdy}\\\\ \displaystyle{=\int_{\Omega^{c,1}\times\omega^a}\left\vert\partial_{x_2}  \varphi_1\right\vert^2dxdy+  \int_{\Omega^{c,3}\times\omega^b}\left\vert\partial_{x_2}  \varphi_3\right\vert^2dxdy.
}
\end{array}\end{equation}
As the last integral in \eqref{F17,2019tkk} is concerned, the first two lines in \eqref{F16,2019z1}
and \eqref{J24,2019all} ensure that
\begin{equation}\label{F17,2019tkklss}\begin{array}{lll} 
\displaystyle{\int_{\Omega^{c,2}\times\left(]0,1[\setminus\left(\omega^a\cup\omega^b\right)\right)}   \partial_{y} \varphi_2\partial_{y}\varphi^\star dxdy=\left( \frac{1}{\zeta_1-\zeta_4+1}-\frac{1}{\zeta_2-\zeta_3}\right) \int_{\Omega^{c,2}} 1dx}\\\\
\displaystyle{= \int_{\Omega^{c,2}\times\left(]0,1[\setminus\left(\omega^a\cup\omega^b\right)\right)}\left\vert\partial_{y}  \varphi_2\right\vert^2dxdy.}
\end{array}\end{equation}
Finally, \eqref{F15,2019energyconv} follows from \eqref{F17,2019tkk}, \eqref{F17,2019tkkha}, and \eqref{F17,2019tkklss}.
\end{proof}

Proposition \ref{Proposizione Monda3},   Proposition \ref{Proposizione Eliquis1}, and Proposition \ref{encon2019} provide the following corrector results.
\begin{Proposition}\label{F19,2019corrres} For every $\varepsilon$, let $\varphi_\varepsilon$ be the unique solution to \eqref{J13,2019weakrescaled} with $\alpha=2$. Moreover,  let   $\varphi_1$,  $\varphi_3$, and $\varphi_2$,  be defined by  \eqref{F12,2019bis},  \eqref{F12,2019}, and \eqref{J24,2019all},  respectively. Then
\begin{equation}\label{F19,2019cr1}\begin{array}{lll} \displaystyle{\lim_{\varepsilon\rightarrow0}
\int_{\Omega^{c,1}_\varepsilon}\left( \left\vert  \varepsilon^2\partial_{x_1}  \varphi_\varepsilon \right\vert^2+\left\vert \partial_{x_2}  \varphi_\varepsilon (x)-\left(\partial_{x_2}\varphi_1 \right)\left(\frac{x_1}{\varepsilon}\right)\right\vert ^2\right)dx=0,
}\end{array}\end{equation}
\begin{equation}\label{F19,2019cr3}\begin{array}{lll} \displaystyle{\lim_{\varepsilon\rightarrow0}
\int_{\Omega^{c,3}_\varepsilon}\left( \left\vert  \varepsilon^2\partial_{x_1}  \varphi_\varepsilon \right\vert^2+\left\vert \partial_{x_2}  \varphi_\varepsilon (x)-\left(\partial_{x_2}\varphi_3 \right)\left(\frac{x_1}{\varepsilon}\right)\right\vert ^2\right)dx=0,
}\end{array}\end{equation}
and
\begin{equation}\label{F19,2019cr2}\begin{array}{lll} \displaystyle{\lim_{\varepsilon\rightarrow0}
\int_{\Omega^{c,2}_\varepsilon}\left( \left\vert  \varepsilon\partial_{x_1}  \varphi_\varepsilon -\left(\partial_y\varphi_2\right)\left(\frac{x_1}{\varepsilon}\right)\right\vert^2+\left\vert \varepsilon\partial_{x_2}  \varphi_\varepsilon (x)\right\vert ^2\right)dx=0.
}\end{array}\end{equation}
\end{Proposition}
\begin{proof} One has
\begin{equation}\nonumber\begin{array}{lll} \displaystyle{
\int_{\Omega^{c,1}_\varepsilon}\left( \left\vert  \varepsilon^2\partial_{x_1}  \varphi_\varepsilon \right\vert^2+\left\vert \partial_{x_2}  \varphi_\varepsilon (x)-\left(\partial_{x_2}\varphi_1 \right)\left(\frac{x_1}{\varepsilon}\right)\right\vert ^2\right)dx
}\\\\ 
 \displaystyle{+
\int_{\Omega^{c,3}_\varepsilon}\left( \left\vert  \varepsilon^2\partial_{x_1}  \varphi_\varepsilon \right\vert^2+\left\vert \partial_{x_2}  \varphi_\varepsilon (x)-\left(\partial_{x_2}\varphi_3 \right)\left(\frac{x_1}{\varepsilon}\right)\right\vert ^2\right)dx
}\\\\
 \displaystyle{+
\int_{\Omega^{c,2}_\varepsilon}\left( D_\varepsilon\left\vert  \varepsilon\partial_{x_1}  \varphi_\varepsilon -\left(\partial_y\varphi_2\right)\left(\frac{x_1}{\varepsilon}\right)\right\vert^2+D_\varepsilon^{-1}\left\vert \varepsilon\partial_{x_2}  \varphi_\varepsilon (x)\right\vert ^2\right)dx=}
\end{array}\end{equation}
\begin{equation}\nonumber\begin{array}{lll} \displaystyle{\int_{\Omega^{c,1}_\varepsilon\cup \Omega^{c,3}_\varepsilon}\left( \left\vert  \varepsilon^2\partial_{x_1}  \varphi_\varepsilon \right\vert^2+\left\vert\partial_{x_2}  \varphi_\varepsilon \right\vert^2\right)dx+
\int_{\Omega^{c,2}_\varepsilon}\left(  D_\varepsilon  \left\vert  \varepsilon\partial_{x_1}  \varphi_\varepsilon \right\vert^2+D^{-1}_\varepsilon \left\vert  \varepsilon\partial_{x_2}  \varphi_\varepsilon \right\vert^2\right)dx}\\\\
\displaystyle{+\int_{\Omega^{c,1}}\left(\left\vert\left(\partial_{x_2}\varphi_1 \right)\left(\frac{x_1}{\varepsilon}\right)\right\vert ^2-2\partial_{x_2}  \widehat{\varphi_{\varepsilon,1}} (x)\left(\partial_{x_2}\varphi_1 \right)\left(\frac{x_1}{\varepsilon}\right)\right)dx}
\\\\
\displaystyle{+\int_{\Omega^{c,3}}\left(\left\vert\left(\partial_{x_2}\varphi_3 \right)\left(\frac{x_1}{\varepsilon}\right)\right\vert ^2-2\partial_{x_2}  \widetilde{\varphi_{\varepsilon,3}} (x)\left(\partial_{x_2}\varphi_3 \right)\left(\frac{x_1}{\varepsilon}\right)\right)dx}
\\\\
\displaystyle{+D_\varepsilon\int_{\Omega^{c,2}}\left(\left\vert\left(\partial_{y}\varphi_2 \right)\left(\frac{x_1}{\varepsilon}\right)\right\vert ^2-2\varepsilon\partial_{x_1}  \overline{\varphi_{\varepsilon,2}} (x)\left(\partial_{y}\varphi_2 \right)\left(\frac{x_1}{\varepsilon}\right)\right)dx,\quad\forall\varepsilon.}
\end{array}\end{equation}
where $ \widehat{\varphi_{\varepsilon,1}}$, $\widetilde{\varphi_{\varepsilon,3}}$, and  $\overline{\varphi_{\varepsilon,2}} $ are defined by    \eqref{Eliquis4}, \eqref{Eliquis2},  and  \eqref{F13,2019}, respectively. Passing to the limit, as $\varepsilon\rightarrow 0$, in this equality and using  Proposition \ref{Proposizione Monda3},   Proposition \ref{Proposizione Eliquis1}, Proposition \ref{encon2019}, and \eqref{J15,2019} provide
\begin{equation}\nonumber\begin{array}{lll} \displaystyle{\lim_{\varepsilon\rightarrow0}\bigg[
\int_{\Omega^{c,1}_\varepsilon}\left( \left\vert  \varepsilon^2\partial_{x_1}  \varphi_\varepsilon \right\vert^2+\left\vert \partial_{x_2}  \varphi_\varepsilon (x)-\left(\partial_{x_2}\varphi_1 \right)\left(\frac{x_1}{\varepsilon}\right)\right\vert ^2\right)dx
}\\\\ 
 \displaystyle{+
\int_{\Omega^{c,3}_\varepsilon}\left( \left\vert  \varepsilon^2\partial_{x_1}  \varphi_\varepsilon \right\vert^2+\left\vert \partial_{x_2}  \varphi_\varepsilon (x)-\left(\partial_{x_2}\varphi_3 \right)\left(\frac{x_1}{\varepsilon}\right)\right\vert ^2\right)dx
}\\\\
 \displaystyle{+
\int_{\Omega^{c,2}_\varepsilon}\left( D_\varepsilon\left\vert  \varepsilon\partial_{x_1}  \varphi_\varepsilon -\left(\partial_y\varphi_2\right)\left(\frac{x_1}{\varepsilon}\right)\right\vert^2+D_\varepsilon^{-1}\left\vert \varepsilon\partial_{x_2}  \varphi_\varepsilon (x)\right\vert ^2\right)dx\bigg]=0,
}\end{array}\end{equation}
which implies \eqref{F19,2019cr1} thanks to \eqref{J15,2019}.\end{proof}

\subsection{Proof of Theorem \ref{main theoremapril24,2019} with $\alpha=2$\label{proofmmmmaitttheore}}

\begin{proof} Proposition \ref{PropF25,2019}  with $\alpha=2$ provides that for every $\varepsilon$
\begin{equation}\label{J27,2019*}\begin{array}{ll}\displaystyle{
\int_{\Gamma^a_{\varepsilon,2}} |\nabla\varepsilon^2 \phi_{\varepsilon} |^{2}\nu_2 ds}\\\\ \displaystyle{=-\int_{\Omega^{c,1}_\varepsilon\cup \Omega^{c,3}_\varepsilon} \partial _{x_{2}}
\varphi^\star_{\varepsilon }\left \vert\varepsilon^2\partial_{x_1}  \varphi_\varepsilon\right\vert^2dx+\int_{\Omega^{c,1}_\varepsilon} \partial _{x_{2}}
\varphi^\star_{\varepsilon }\left\vert\partial_{x_2}  \varphi_\varepsilon\right\vert^2dx+\int_{\Omega^{c,3}_\varepsilon} \partial _{x_{2}}
\varphi^\star_{\varepsilon }\left\vert\partial_{x_2}  \varphi_\varepsilon\right\vert^2dx}\\\\
\displaystyle{+2\varepsilon^4
\int_{\Omega^{c,1}_\varepsilon\cup \Omega^{c,3}_\varepsilon}\partial_{x_2}
\varphi_{\varepsilon } \partial_{x_1}
\varphi^\star_{\varepsilon }\partial_{x_1}
\varphi_{\varepsilon }dx}
\\\\
\displaystyle{+\varepsilon^2
\int_{\Omega^{c,2}_\varepsilon} \left(-\partial _{x_{2}}
\varphi^\star_{\varepsilon }\left(\left \vert\varepsilon\partial_{x_1}  \varphi_\varepsilon\right\vert^2-\frac{1}{D_\varepsilon^2}\left\vert\varepsilon\partial_{x_2}  \varphi_\varepsilon\right\vert^2\right)+2\varepsilon\partial_{x_2}
\varphi_{\varepsilon }\partial_{x_1}
\varphi^\star_{\varepsilon }\varepsilon\partial_{x_1}
\varphi_{\varepsilon }\right)dx.}
\end{array}\end{equation}

As the first integral   in the right-hand side of \eqref{J27,2019*} is concerned,  \eqref{F16,2019z}-\eqref{F16,2019zq}, \eqref{F19,2019cr1}, and  \eqref{F19,2019cr3}  provide that
\begin{equation}\label{J26,2019I}\begin{array}{ll}\displaystyle{\left\vert \int_{\Omega^{c,1}_\varepsilon\cup \Omega^{c,3}_\varepsilon} \partial _{x_{2}}
\varphi^\star_{\varepsilon }\left \vert\varepsilon^2\partial_{x_1}  \varphi_\varepsilon\right\vert^2 dx\right\vert}\\\\ \displaystyle{\leq\Vert  \partial _{x_{2}}
\varphi^\star\Vert_{L^\infty([0,1] \times[l_1,l_2])}  \int_{\Omega^{c,1}_\varepsilon\cup \Omega^{c,3}_\varepsilon}\left \vert\varepsilon^2\partial_{x_1}  \varphi_\varepsilon\right\vert^2 dx\rightarrow 0, }
\end{array}\end{equation}
as $\varepsilon\rightarrow 0$.

As the second integral   in the right-hand side of \eqref{J27,2019*} is concerned, one has
\begin{equation}\label{J26,2019II}\begin{array}{ll}\displaystyle{
\int_{\Omega^{c,1}_\varepsilon} \partial _{x_{2}}
\varphi^\star_{\varepsilon }\left\vert\partial_{x_2}  \varphi_\varepsilon\right\vert^2dx=
\int_{\Omega^{c,1}_\varepsilon} \partial _{x_{2}}
\varphi^\star_{\varepsilon }\left\vert\partial_{x_2}  \varphi_\varepsilon-\left(\partial_{x_2}\varphi_1 \right)\left(\frac{x_1}{\varepsilon}\right)+\left(\partial_{x_2}\varphi_1 \right)\left(\frac{x_1}{\varepsilon}\right)\right\vert^2dx}\\\\
\displaystyle{= \int_{\Omega^{c,1}_\varepsilon} \partial _{x_{2}}
\varphi^\star_{\varepsilon }\left\vert\partial_{x_2}  \varphi_\varepsilon-\left(\partial_{x_2}\varphi_1 \right)\left(\frac{x_1}{\varepsilon}\right)\right\vert^2dx
+\int_{\Omega^{c,1}_\varepsilon} \partial _{x_{2}}
\varphi^\star_{\varepsilon }\left\vert\left(\partial_{x_2}\varphi_1 \right)\left(\frac{x_1}{\varepsilon}\right)\right\vert^2dx}\\\\
\displaystyle{+2\int_{\Omega^{c,1}_\varepsilon} \partial _{x_{2}}
\varphi^\star_{\varepsilon }\left(\partial_{x_2}  \varphi_\varepsilon-\left(\partial_{x_2}\varphi_1 \right)\left(\frac{x_1}{\varepsilon}\right)\right)\left(\partial_{x_2}\varphi_1 \right)\left(\frac{x_1}{\varepsilon}\right)dx,\quad\forall\varepsilon.
}
\end{array}\end{equation}
where $\varphi_1$ is defined in \eqref{F12,2019bis}. Moreover,
 \eqref{F16,2019z}-\eqref{F16,2019zq},  \eqref{F12,2019bis}, \eqref{F19,2019cr1}, and  \eqref{F19,2019cr3}  provide
\begin{equation}\label{J26,2019IIbis}\left\{\begin{array}{ll}\displaystyle{\left\vert  \int_{\Omega^{c,1}_\varepsilon} \partial _{x_{2}}
\varphi^\star_{\varepsilon }\left\vert\partial_{x_2}  \varphi_\varepsilon-\left(\partial_{x_2}\varphi_1 \right)\left(\frac{x_1}{\varepsilon}\right)\right\vert^2dx\right\vert
}\\\\
 \displaystyle{\leq\Vert  \partial _{x_{2}}
\varphi^\star\Vert_{L^\infty([0,1] \times[l_1,l_2])}   \int_{\Omega^{c,1}_\varepsilon} \left\vert\partial_{x_2}  \varphi_\varepsilon-\left(\partial_{x_2}\varphi_1 \right)\left(\frac{x_1}{\varepsilon}\right)\right\vert^2dx\rightarrow 0, }\\\\
\displaystyle{\int_{\Omega^{c,1}_\varepsilon} \partial _{x_{2}}
\varphi^\star_{\varepsilon }\left\vert\left(\partial_{x_2}\varphi_1 \right)\left(\frac{x_1}{\varepsilon}\right)\right\vert^2dx=\int_{\Omega^{c,1}} \partial _{x_{2}}
\varphi^\star\left(\frac{x_1}{\varepsilon},x_2\right)\left\vert\left(\partial_{x_2}\varphi_1 \right)\left(\frac{x_1}{\varepsilon}\right)\right\vert^2dx}\\\\
\displaystyle{\rightarrow \int_{\Omega^{c,1}\times\omega^a} \partial _{x_{2}}
\varphi^\star\left(y,x_2\right)dxdy=\hbox{meas}(\omega^a) L,}\\\\
\displaystyle{2\left\vert\int_{\Omega^{c,1}_\varepsilon} \partial _{x_{2}}
\varphi^\star_{\varepsilon }\left(\partial_{x_2}  \varphi_\varepsilon-\left(\partial_{x_2}\varphi_1 \right)\left(\frac{x_1}{\varepsilon}\right)\right)\left(\partial_{x_2}\varphi_1 \right)\left(\frac{x_1}{\varepsilon}\right)dx\right\vert}
\\\\
 \displaystyle{\leq2\Vert  \partial _{x_{2}}
\varphi^\star\Vert_{L^\infty([0,1] \times[l_1,l_2])} \Vert  \partial _{x_{2}}
\varphi_1\Vert_{L^\infty([0,1] )}   \int_{\Omega^{c,1}_\varepsilon} \left\vert\partial_{x_2}  \varphi_\varepsilon-\left(\partial_{x_2}\varphi_1 \right)\left(\frac{x_1}{\varepsilon}\right)\right\vert dx\rightarrow 0, }
\end{array}\right.\end{equation}
 as $\varepsilon\rightarrow 0$. Then, combining \eqref{J26,2019II} and \eqref{J26,2019IIbis} gives
 \begin{equation}\label{J27,2019+}\begin{array}{ll}\displaystyle{\lim_{\varepsilon\rightarrow 0}
\int_{\Omega^{c,1}_\varepsilon} \partial _{x_{2}}
\varphi^\star_{\varepsilon }\left\vert\partial_{x_2}  \varphi_\varepsilon\right\vert^2dx=
\hbox{meas}(\omega^a)L. }
\end{array}\end{equation}
Similarly, one proves that
 \begin{equation}\label{J27,2019+-}\begin{array}{ll}\displaystyle{\lim_{\varepsilon\rightarrow 0}
\int_{\Omega^{c,3}_\varepsilon} \partial _{x_{2}}
\varphi^\star_{\varepsilon }\left\vert\partial_{x_2}  \varphi_\varepsilon\right\vert^2dx=
\hbox{meas}(\omega^b)L. }
\end{array}\end{equation}

As the fourth integral   in the right-hand side of \eqref{J27,2019*} is concerned, \eqref{F16,2019z}-\eqref{F16,2019zq},   and the first two estimates in  \eqref{J18,2019estimatesalfa=2}  provide
\begin{equation}\label{J27,2019IIbiscx}\begin{array}{ll}
\displaystyle{\left\vert
2\varepsilon^4
\int_{\Omega^{c,1}_\varepsilon\cup \Omega^{c,3}_\varepsilon}\partial_{x_2}
\varphi_{\varepsilon } \partial_{x_1}
\varphi^\star_{\varepsilon }\partial_{x_1}
\varphi_{\varepsilon }dx\right\vert
}\\\\
\leq
2\varepsilon\Vert  \partial _{x_{1}}
\varphi^\star\Vert_{L^\infty([0,1] \times[l_1,l_2])} 
\Vert \varepsilon^2 \partial _{x_{1}}
\varphi_\varepsilon\Vert_{L^2(\Omega^{c,1}_\varepsilon\cup \Omega^{c,3}_\varepsilon)}
\Vert \partial _{x_{2}}
\varphi_\varepsilon\Vert_{L^2(\Omega^{c,1}_\varepsilon\cup \Omega^{c,3}_\varepsilon)}\rightarrow0,
\end{array}\end{equation}
as $\varepsilon\rightarrow 0$.

As the last integral   in the right-hand side of \eqref{J27,2019*} is concerned, \eqref{J15,2019},  \eqref{F16,2019z}-\eqref{F16,2019zq},  and the last estimate in  \eqref{J18,2019estimatesalfa=2}  provide
\begin{equation}\label{J27,2019IIbiscxfhf}\begin{array}{ll}
\displaystyle{\left\vert \varepsilon^2
\int_{\Omega^{c,2}_\varepsilon} \left(-\partial _{x_{2}}
\varphi^\star_{\varepsilon }\left(\left \vert\varepsilon\partial_{x_1}  \varphi_\varepsilon\right\vert^2-\frac{1}{D_\varepsilon^2}\left\vert\varepsilon\partial_{x_2}  \varphi_\varepsilon\right\vert^2\right)+2\varepsilon\partial_{x_2}
\varphi_{\varepsilon }\partial_{x_1}
\varphi^\star_{\varepsilon }\varepsilon\partial_{x_1}
\varphi_{\varepsilon }\right)dx\right\vert}\\\\
\displaystyle{\leq\bigg[\varepsilon^2\Vert  \partial _{x_{2}}
\varphi^\star\Vert_{L^\infty([0,1] \times[l_1,l_2])} \int_{\Omega^{c,2}_\varepsilon} \left(\left \vert\varepsilon\partial_{x_1}  \varphi_\varepsilon\right\vert^2+\frac{1}{D_\varepsilon^2}\left\vert\varepsilon\partial_{x_2}  \varphi_\varepsilon\right\vert^2\right)dx
}\\\\
\displaystyle{+2\varepsilon \Vert  \partial _{x_{1}}
\varphi^\star\Vert_{L^\infty([0,1] \times[l_1,l_2])}
\Vert \varepsilon \partial _{x_{1}}
\varphi_\varepsilon\Vert_{L^2(\Omega^{c,2}_\varepsilon)}
\Vert \varepsilon\partial _{x_{2}}
\varphi_\varepsilon\Vert_{L^2(\Omega^{c,2}_\varepsilon)}\bigg]\rightarrow0,
}
\end{array}\end{equation}
as $\varepsilon\rightarrow 0$.

Finally, passing to the limit, as $\varepsilon$ tends to zero, in \eqref{J27,2019*} and using \eqref{J26,2019I}, \eqref{J27,2019+}, \eqref{J27,2019+-}, \eqref{J27,2019IIbiscx}, and \eqref{J27,2019IIbiscxfhf} give \eqref{F25,2019has} when $\alpha=2$.
\end{proof}

\section{The case $\alpha>2$\label{sketched}}
In the case $\alpha>2$, the proof of Theorem \ref{main theoremapril24,2019}  will be just sketched.
\subsection{{\it A priori} estimates }
Proposition \ref{J18,2019Proposition} immediately implies the following result.
\begin{Corollary} For every $\varepsilon$, let $\varphi_\varepsilon$ be the unique solution to \eqref{J13,2019weakrescaled} with $\alpha>2$. Then,
\begin{equation}\label{J18,2019estimatesalfa=2Aprile 62019}\exists c\in]0,+\infty[\quad:\quad\left\{\begin{array}{lll}\displaystyle{ \Vert\varepsilon^\alpha\partial_{x_1}  \varphi_\varepsilon\Vert_{L^2(\Omega^{c,1}_\varepsilon\cup \Omega^{c,3}_\varepsilon)}\leq  c,}\\\\
\displaystyle{  \Vert\partial_{x_2}  \varphi_\varepsilon \Vert_{L^2(\Omega^{c,1}_\varepsilon\cup \Omega^{c,3}_\varepsilon)}\leq  c, 
}\\\\ \displaystyle{ \Vert\varepsilon^{\frac{\alpha}{2}}\nabla  \varphi_\varepsilon \Vert_{L^2(\Omega^{c,2}_\varepsilon)}\leq  c,}
\end{array}\right.\quad\forall\varepsilon.\end{equation}
\end{Corollary}
This result provides the   following {\it a priori} estimate.
\begin{Proposition} \label{J21,2019prop6Aprile2019} For every $\varepsilon$, let $\varphi_\varepsilon$ be the unique solution to \eqref{J13,2019weakrescaled} with $\alpha>2$. Then,
\begin{equation}\label{J18,2019las6Aprile2019t}\exists c\in]0,+\infty[\quad:\quad\left\{\begin{array}{lll}\displaystyle{ \Vert \varphi_\varepsilon  \Vert_{L^2(\Omega^{c,1}_\varepsilon\cup \Omega^{c,3}_\varepsilon)}\leq c, }\\\\ \left\Vert \varepsilon^{\frac{\alpha-2}{2}} \varphi_\varepsilon  \right\Vert_{L^2(\Omega^{c,2}_\varepsilon)}\leq c,
\end{array}\right.\quad\forall\varepsilon.\end{equation}
\end{Proposition} 
\begin{proof} The Dirichlet boundary condition of $\varphi_\varepsilon $ on $\Gamma_\varepsilon$ and the second estimate in \eqref{J18,2019estimatesalfa=2Aprile 62019} provide  the first estimate in \eqref{J18,2019las6Aprile2019t}.

Arguing as in the proof of Proposition \ref{J21,2019prop} gives
\begin{equation}\label{J19,2019Aprile62019}
\Vert \varepsilon^{\frac{\alpha-2}{2}}  \varphi_\varepsilon  \Vert^2_{L^2(\Omega^{c,2}_\varepsilon)}\leq 
2(l_2-l_1)\varepsilon^{\alpha-2} +2
\left \Vert \varepsilon^{\frac{\alpha}{2}} \partial_{x_1}\varphi_\varepsilon \right\Vert^2_{L^2(\Omega^{c,2}_\varepsilon)},\quad\forall\varepsilon,
\end{equation}
which implies the second estimate in \eqref{J18,2019las6Aprile2019t}, thanks to the third estimate in \eqref{J18,2019estimatesalfa=2Aprile 62019}.
\end{proof}

\subsection{Weak convergence results}
The next proposition is devoted to  studying the limit  in $\Omega^{c,2}$, as $\varepsilon$ tends to zero,  of problem \eqref{J13,2019weakrescaled} with $\alpha>2$.
\begin{Proposition} \label{Proposizione Monda37Aprile2019}For every $\varepsilon$, let $\varphi_\varepsilon$ be the unique solution to \eqref{J13,2019weakrescaled} with $\alpha>2$ and  let  $\overline{\varphi_{\varepsilon,2}}$,
 be defined by \eqref{F13,2019}.
Then, 
\begin{equation}\label{Monda1bisterforse7Aprile2019}\left\{\begin{array}{ll}\varepsilon^{\frac{\alpha-2}{2}}\overline{\varphi_{\varepsilon,2}}\hbox{ two scale converges to } 0,\\\\
\varepsilon^{\frac{\alpha}{2}}\partial_{x_1}\overline{\varphi_{\varepsilon,2}}\hbox{ two scale converges to }0,\\\\
\varepsilon^{\frac{\alpha}{2}}\partial_{x_2}\overline{\varphi_{\varepsilon,2}}\hbox{ two scale converges to }0,\end{array}\right.\end{equation}
as $\varepsilon$ tends to zero.
\end{Proposition}
\begin{proof}
The second estimate in  \eqref{J18,2019las6Aprile2019t} and the third estimate in \eqref{J18,2019estimatesalfa=2Aprile 62019} ensure the existence of a subsequence of $\{\varepsilon\}$, still denoted by $\{\varepsilon\}$,  and $u_2 \in L^2\left(\Omega^{c,2}, H^1_{\hbox{per}}(]0,1[)\right)$   (in possible dependence on the subsequence) such that
 \begin{equation}\label{Monda1bister7Aprile2019}\left\{\begin{array}{ll}\varepsilon^{\frac{\alpha-2}{2}}\overline{\varphi_{\varepsilon,2}}\hbox{ two scale converges to } u_2,\\\\
\varepsilon^{\frac{\alpha}{2}}\partial_{x_1}\overline{\varphi_{\varepsilon,2}}\hbox{ two scale converges to }\partial_{y} u_2,\\\\
\varepsilon^{\frac{\alpha}{2}}\partial_{x_2}\overline{\varphi_{\varepsilon,2}}\hbox{ two scale converges to }0,\end{array}\right.\end{equation}
as $\varepsilon$ tends to zero.

Arguing as in the proof of Proposition \ref{Proposizione Monda3}, one obtains 
\begin{equation}\label{J23,2019fgd7Aprile2019} u_2=0, \hbox { a.e. in } \Omega^{c,2}\times\left(\omega^a\cup\omega^b\right).\end{equation}

Passing to the limit, as $\varepsilon$ tends to zero, in \eqref{J13,2019weakrescaled} with $\alpha>2$ and with test functions
 $\displaystyle{\psi=\varepsilon^{\frac{\alpha}{2}+1}\chi_1(x_1,x_2)\chi_2\left(\frac{x_1}{\varepsilon}\right)}$, where $\chi_1\in C_0^\infty\left(\Omega^{c,2}\right)$ and $\chi_2\in H^1_{\hbox{per}}\left(]0,1[\right)$ such that $\chi_2=0$ in $\omega^a\cup\omega^b$, and using  \eqref{J15,2019} and  the second and third limits in   \eqref{Monda1bister7Aprile2019} provide
 that, for a.e. $(x_1,x_2)$ in $ \Omega^{c,2}$,
\begin{equation}\label{J23,2019seraIJ247Aprile2019}\begin{array}{ll}\displaystyle{ \int_{]0,1[\setminus\left(\omega^a\cup\omega^b\right)}\partial_{y}u_2(x_1,x_2,y)\partial_{y}\chi_2(y)dy=0, }\\\\  \forall \chi_2\in H^1_{\hbox{per}}\left(]0,1[\right)\,\,:\,\, \chi_2=0,\hbox{ in } \omega^a\cup\omega^b.
\end{array}
\end{equation}

Problem \eqref{J23,2019fgd7Aprile2019} and \eqref{J23,2019seraIJ247Aprile2019} is equivalent to the following problem independent of $(x_1,x_2)$
\begin{equation}\left\{\begin{array}{ll}\partial^2_{y^2}u_2=0,\hbox{ in }]0,1[\setminus\left(\omega^a\cup\omega^b\right),\\\\
u_2=0, \hbox {  in } \omega^a\cup\omega^b,\\\\
  u_2(0)=u_2(1),\\\\
   \partial_y u_2(0)=\partial_yu_2(1),
\end{array}\right.\end{equation}
which admits  $u_2=0$ as unique solution.
Consequently, limits in  \eqref{Monda1bister7Aprile2019} hold for the whole sequence and \eqref{Monda1bisterforse7Aprile2019} is satisfied.
\end{proof}

The next proposition is devoted to  studying the limit  in $\Omega^{c,3}$ and in $\Omega^{c,1}$, as $\varepsilon$ tends to zero,  of problem \eqref{J13,2019weakrescaled} with $\alpha>2$.

\begin{Proposition} \label{Proposizione Eliquis18Aprile2019}
For every $\varepsilon$, let $\varphi_\varepsilon$ be the unique solution to \eqref{J13,2019weakrescaled} with $\alpha>2$  and let  
$\widetilde{\varphi_{\varepsilon,3}}$ and $\widehat{\varphi_{\varepsilon,1}}$ be defined by \eqref{Eliquis2} and  \eqref{Eliquis4}, respectively. Moreover,  let   $\varphi_3$ and $\varphi_1$ be defined by  \eqref{F12,2019},and \eqref{F12,2019bis}, respectively. Then,
\begin{equation}\label{Monda1bisterforseter8Aprile2019}\left\{\begin{array}{ll}\widetilde{\varphi_{\varepsilon,3}}\hbox{ two scale converges to } \varphi_3,\\\\
\partial_{x_2}\widetilde{\varphi_{\varepsilon,3}}\hbox{ two scale converges to }\partial_{x_2} \varphi_3,\end{array}\right.\end{equation}
and
\begin{equation}\label{Monda1bisterforsequater8Aprile2019}\left\{\begin{array}{ll}\widehat{\varphi_{\varepsilon,1}}\hbox{ two scale converges to } \varphi_1,\\\\
\partial_{x_2}\widehat{\varphi_{\varepsilon,1}}\hbox{ two scale converges to }\partial_{x_2} \varphi_1,\end{array}\right.\end{equation}
as $\varepsilon$ tends to zero.
\end{Proposition}
\begin{proof} One can repeat the proof of  Proposition \ref{Proposizione Eliquis1}, by making attention to use equation \eqref{J13,2019weakrescaled} with $\alpha>2$ instead of  $\alpha=2$, and to multiply  the test functions by $\varepsilon^\alpha$  instead of $\varepsilon^2$ when it occurs. Really, in this case the proof is simpler  than the proof of  Proposition \ref{Proposizione Eliquis1} due to the fact that the second limit in \eqref{Monda1bisterforse7Aprile2019} is zero. \end{proof}

The following result is an immediate consequence of Proposition \ref{Proposizione Monda37Aprile2019} and Proposition \ref{Proposizione Eliquis18Aprile2019}.
\begin{Corollary}
For every $\varepsilon$, let $\varphi_\varepsilon$ be the unique solution to \eqref{J13,2019weakrescaled} with $\alpha>2$  and let  $\overline{\varphi_{\varepsilon,2}}$,
$\widetilde{\varphi_{\varepsilon,3}}$, and $\widehat{\varphi_{\varepsilon,1}}$ be defined by \eqref{F13,2019}, \eqref{Eliquis2}, and  \eqref{Eliquis4}, respectively. Moreover,  let  $\varphi_3$  and $\varphi_1$ be defined by  \eqref{F12,2019} and \eqref{F12,2019bis}, respectively. Then
\begin{equation} \nonumber\varepsilon^{\frac{\alpha-2}{2}}\overline{\varphi_{\varepsilon,2}}\rightharpoonup 0,\quad\varepsilon^{\frac{\alpha}{2}}\partial_{x_1}\overline{\varphi_{\varepsilon,2}}\rightharpoonup0,\quad \varepsilon^{\frac{\alpha}{2}}\partial_{x_2}\overline{\varphi_{\varepsilon,2}}\rightharpoonup0, \hbox{ weakly in }L^2(\Omega^{c,2}),
\end{equation}
\begin{equation} \nonumber\widetilde{\varphi_{\varepsilon,3}}\rightharpoonup (x_2-l_2)\hbox{meas}(\omega^b)+1, \quad\partial_{x_2}\widetilde{\varphi_{\varepsilon,3}}\rightharpoonup\hbox{meas}(\omega^b),\hbox{ weakly in }L^2(\Omega^{c,3}),
\end{equation}
and
\begin{equation} \nonumber\widehat{\varphi_{\varepsilon,1}}\rightharpoonup (x_2-l_1)\hbox{meas}(\omega^a),\quad  \quad\partial_{x_2}\widehat{\varphi_{\varepsilon,1}}\rightharpoonup\hbox{meas}(\omega^a),\hbox{ weakly in }L^2(\Omega^{c,1}),
\end{equation}
as $\varepsilon$ tends to zero.
\end{Corollary}

\subsection{Corrector results}
Arguing as in Proposition \ref{encon2019}, one obtains the following energies convergence.
\begin{Proposition}\label{encon2019Pasqua2019} For every $\varepsilon$, let $\varphi_\varepsilon$ be the unique solution to \eqref{J13,2019weakrescaled} with $\alpha>2$. Moreover,  let   $\varphi_1$ and   $\varphi_3$  be defined by  \eqref{F12,2019bis} and  \eqref{F12,2019},  respectively. Then
\begin{equation}\nonumber\begin{array}{lll} 
\displaystyle{\lim_{\varepsilon\rightarrow 0}\bigg[\int_{\Omega^{c,1}_\varepsilon\cup \Omega^{c,3}_\varepsilon}\left( \left\vert  \varepsilon^\alpha\partial_{x_1}  \varphi_\varepsilon \right\vert^2+\left\vert\partial_{x_2}  \varphi_\varepsilon \right\vert^2\right)+
\int_{\Omega^{c,2}_\varepsilon}\left(  D_\varepsilon \left\vert\varepsilon^{\frac{\alpha}{2}}\partial_{x_1}  \varphi_\varepsilon \right\vert^2+D^{-1}_\varepsilon\left\vert\varepsilon^{\frac{\alpha}{2}}\partial_{x_2}  \varphi_\varepsilon \right\vert^2 \right)dx\bigg]
}\\\\
\displaystyle{=
\int_{\Omega^{c,1}\times\omega^a}\left\vert\partial_{x_2}  \varphi_1\right\vert^2dxdy+\int_{\Omega^{c,3}\times\omega^b}\left\vert\partial_{x_2}  \varphi_3\right\vert^2dxdy.}
\end{array}\end{equation}
\end{Proposition}
By arguing as in Proposition \ref{F19,2019corrres}, Proposition \ref{Proposizione Monda37Aprile2019},   Proposition \ref{Proposizione Eliquis18Aprile2019}, and Proposition \ref{encon2019Pasqua2019} provide the following corrector results.
\begin{Proposition}\label{F19,2019corrresPasqua2019} For every $\varepsilon$, let $\varphi_\varepsilon$ be the unique solution to \eqref{J13,2019weakrescaled} with $\alpha>2$. Moreover,  let   $\varphi_1$ and   $\varphi_3$  be defined by  \eqref{F12,2019bis} and  \eqref{F12,2019},  respectively. Then
\begin{equation}\nonumber\begin{array}{lll} \displaystyle{\lim_{\varepsilon\rightarrow0}
\int_{\Omega^{c,1}_\varepsilon}\left( \left\vert  \varepsilon^\alpha\partial_{x_1}  \varphi_\varepsilon \right\vert^2+\left\vert \partial_{x_2}  \varphi_\varepsilon (x)-\left(\partial_{x_2}\varphi_1 \right)\left(\frac{x_1}{\varepsilon}\right)\right\vert ^2\right)dx=0,
}\end{array}\end{equation}
\begin{equation}\nonumber\begin{array}{lll} \displaystyle{\lim_{\varepsilon\rightarrow0}
\int_{\Omega^{c,3}_\varepsilon}\left( \left\vert  \varepsilon^\alpha\partial_{x_1}  \varphi_\varepsilon \right\vert^2+\left\vert \partial_{x_2}  \varphi_\varepsilon (x)-\left(\partial_{x_2}\varphi_3 \right)\left(\frac{x_1}{\varepsilon}\right)\right\vert ^2\right)dx=0,
}\end{array}\end{equation}
and
\begin{equation}\nonumber\begin{array}{lll} \displaystyle{\lim_{\varepsilon\rightarrow0}
\int_{\Omega^{c,2}_\varepsilon}\left( \left\vert  \varepsilon^{\frac{\alpha}{2}}\partial_{x_1}  \varphi_\varepsilon \right\vert^2+\left\vert \varepsilon^{\frac{\alpha}{2}}\partial_{x_2}  \varphi_\varepsilon (x)\right\vert ^2\right)dx=0.
}\end{array}\end{equation}
\end{Proposition}

Finally, using Proposition \ref{F19,2019corrresPasqua2019}, the proof of Theorem  \ref{main theoremapril24,2019}
with  $\alpha>2$ follows the same outline of  the proof of Theorem  \ref{main theoremapril24,2019} with $\alpha=2$.

\section*{Acknowledgments}

 The authors thank the
 "Gruppo Nazionale per l'Analisi Matematica, la Probabilit\`a e le loro Applicazioni (GNAMPA)"
 of the "Istituto Nazionale di Alta Matematica (INdAM)" (Italy),  the "Universit\'e de Franche-Comt\'e"  (France), and  the  competitive funding program  for interdisciplinary
 research  of  "CNRS" (France)  for their financial  support.

\end{document}